\documentclass{amsart}
\usepackage{latexsym}
\usepackage{amsfonts}
\newtheorem{theorem}{Theorem}[section]
\newtheorem{lemma}[theorem]{Lemma}
\newtheorem{corollary}[theorem]{Corollary}
\newtheorem{proposition}[theorem]{Proposition}

\theoremstyle{definition}
\newtheorem{definition}[theorem]{Definition}
\newtheorem{notation}[theorem]{Notation}

\theoremstyle{remark}
\newtheorem{remark}[theorem]{Remark}
\numberwithin{equation}{section}

\newcommand{\sai} {\mbox{$\to \kern -0.50 em \to$}}
\newcommand{\nsai} {\mbox{$\not\to \kern -0.50 em \to$}}

\begin{document}
 
\title[SOME SIMPLIFICATIONS IN BASIC COMPLEX ANALYSIS] 
{SOME SIMPLIFICATIONS IN BASIC COMPLEX ANALYSIS }

\author{Oswaldo Rio Branco de Oliveira}
 
\begin{abstract} {This paper presents very simple and easy integration-free proofs in the context of Weierstrass's theory of functions, of the Maximum and Minimum Modulus Principles and Gutzmer-Parseval Inequalities for polynomials and for functions developable in complex power series at every point in their domains, as well as a trivial proof of the Open Mapping Theorem, an intuitive version of Liouville's Theorem, an easy proof of Weierstrass's Theorem on Double Series, a modest extension of Schwarz's Lemma, and some other related results. It also presents easy proofs of the P\'{o}lya-Szeg\"{o} and P. Erd\"{o}s' Anti-Calculus Proposition, a theorem on saddle points by Bak-Ding-Newman, and the well-known Clunie-Jack Lemma. } 
\end{abstract}
 
\maketitle

\par

\tableofcontents

\section{Introduction}  
The aim of this work is, by employing a ``method'' used in the elementary proof of the Fundamental Theorem of Algebra by de Oliveira ~\cite{OO1} (see also ~\cite{OO2}) and an averaging technique, to give simple, easy, and independent proofs for polynomial versions of a result here named the Gutzmer-Parseval Inequality (by combining the attributions in Burckel~\cite[p.~81]{BU} and Remmert ~\cite[p.~243]{RR1}), the Maximum Modulus Principle, and the Minimum Modulus Principle (also known as Cauchy's Minimum Principle, see Remmert ~\cite[p.~112]{RR2}). This work also aims, through the use of very basic concepts in plane topology, basic results on complex power series, and the two tools already mentioned, to provide extensions of those polynomial results to all power series, in the context of Weierstrass's theory of functions. In addition, this article provides proofs of the Open Mapping Theorem, an Inverse Function Theorem,  Liouville's Theorem (and an extension of it), a theorem by P\'{o}lya-Szeg\"{o} and Erd\"{o}s, a quite recent theorem on saddle points by Bak-Ding-Newman, and the Clunie-Jack Lemma. Some consequences of a Polygonal Mean-Value Property for Polynomials are also proved.

Moreover, this paper proves the easy part of a simplification given by Whyburn of a theorem on power series (independently) demonstrated by Hurwitz, Connell and Porcelli, and Read. Then, through employing the Gutzmer-Parseval Inequality, this article  furnishes a modest extension of Schwarz's Lemma and a rather easy proof for Weierstrass's theorem on double series (see Remmert ~\cite[pp.~250--251]{RR1} and Knopp ~\cite[pp.~430--433]{KK}), a result considered by Weierstrass as the key to convergence theory (see Remmert~\cite[pp.~250--251]{RR1}). In addition, still using the Gutzmer-Parseval inequality, another convergence theorem and Montel's Theorem are proved. Lastly, two results on Laurent series are also proved.

It is interesting to notice that this work provides proofs that do not employ function continuity for the following results: the Gutzmer-Parseval Inequality, Cauchy's Inequalities, Maximum Modulus Principle, Liouville's Theorem, and the Uniqueness Theorem for the coefficients of a power series. 

It is remarked in Conway ~\cite[p.~80]{CON} that ``the Maximum Modulus Theorem ... is far from obvious even for polynomials.'' In Lang ~\cite[p.~84]{LA}, the Maximum Modulus Principle is shown to be a consequence of the Open Mapping Theorem, for which an elaborate proof is given, by applying the theorem on existence of a local compositional inverse $g(w)=\sum b_n(w-w_0)^n$, where $b_n\in \mathbb C$ and $n \in \mathbb N$, for a power series $f(z)=\sum a_n(z-z_0)^n$, where $a_n \in \mathbb C$ and $n\in \mathbb N$, if $f'(z_0)\neq 0$; that is, we have $(g\circ f)(z)=z$ for all $z$ in a neighborhood of $z_0$. In this presentation we will not use this existence theorem. Moreover, Beardon ~\cite[p.~103]{BE} proves the Maximum Modulus Principle for Polynomials by using the Argument Principle.

We recall that a function $f:\Omega \to \mathbb C$, with $\Omega$ an open
subset of $\mathbb C$, is \textit{complex-differentiable}, or \textit{holomorphic}, if $f$ has
complex derivatives $f'(z)=\lim\limits_{h\to 0}\frac{f(z+h) -f(z)}{h}$ at
every point $z\in \Omega$. In \cite{RR1} Remmert pointed out that the goal of Karl Weierstrass was to establish the study of holomorphic functions solely on the basis of power series, without the use of integrals; and although such a methodologically pure path has now been abandoned, modern authors such as Burckel \cite{BU}, Lang \cite{LA}, Bak and Newman \cite{BN}, Remmert \cite{RR1}, and others still stress the importance of the study of power series. A translation of Carath\'{e}odory's opinion is presented by Remmert ~\cite[p.~109]{RR1} as ``Power series are therefore especially convenient because one can compute with them almost
as with polynomials.''

\section{Preliminaries}
Let us denote by $\mathbb N=\{0,1,2,...\}$ the set of all natural numbers, $\mathbb Z$ the set of all integer numbers, $\mathbb Q$ the field of rational numbers, $\mathbb R$ the complete field of real numbers, and $\mathbb C$ the algebraically closed field of complex numbers. Moreover, if $z\in\mathbb C$ then we write $z=x+iy$, where $x=$ Re$(z)\in \mathbb R$ is the real part of $z$, $y=$ Im$(z)\in \mathbb R$ is the imaginary part of $z$, and $i^2=-1$. Given $z=x+iy$ in $\mathbb C$, its conjugate is the complex number $\overline{z} = x-iy$ and its absolute value is the non-negative real number $|z|=\sqrt{z\overline{z}}=\sqrt{x^2 + y^2}$.

The open disk centered at the point $z_0\in \mathbb C$ with radius $r>0$ is the set $D(z_0;r)=\{z \in \mathbb C: |z-z_0|<r\}$. Similarly, the compact disk centered at $z_0$ with radius $r\geq 0$ is the set
$\overline{D}(z_0;r)=\{z \in \mathbb C: |z-z_0|\leq r\}$.

Given $X\subset \mathbb C$, a point $p \in \mathbb C$ is an \textit{ accumulation point} of $X$ if every disk $D(p;r)$, where $r>0$, contains a point of $X$ distinct of $p$.

In this text we will use the following well-known results on power series (see de Oliveira ~\cite{OO3}; see also ~\cite{BN}, ~\cite{BU}, ~\cite{KK}, and ~\cite{LA}): 
\begin{itemize}
\item[$\circ$] Let $(a_n)$ be a sequence of complex numbers. Applying the {\sf Cauchy-Hadamard Formula}, $\rho^{-1} =\limsup\sqrt[n]{|a_n|}$, it follows that if $\rho>0 $, then the power series $f(z)=\sum a_n(z-z_0)^n$ converges uniformly and absolutely on any compact disk $\overline{D}(z_0;r)\subset D(z_0;\rho)$ and diverges at every point $z$ such that $|z-z_0|>\rho$. We call $\rho$ and $D(z_0;\rho)$, the \textit{ radius of convergence} and the \textit{disk of convergence} of the power series, respectively. The function $f$ is continuous on $D(z_0;\rho)$. 

\item[$\circ$] If $f(z) =\sum a_n(z-z_0)^n$ and $g(z)=\sum b_n(z-z_0)^n$ are convergent power series with radii of convergence $\rho_1>0$ and $\rho_2>0$, respectively, and $\lambda\in \mathbb C$, then $\lambda f(z)=\sum \lambda a_n(z-z_0)^n$ is a convergent power series with radius of convergence either equal to $\rho_1$, if $\lambda \neq 0$, or equal to $+\infty$, if $\lambda =0$. Moreover, $f(z)+g(z) =\sum(a_n+b_n)(z-z_0)^n$ and $f(z)g(z)=\sum\limits_{n=0}^{+\infty}\Big(\sum\limits_{j+k=n}
a_jb_k\Big)(z-z_0)^n$ are convergent power series with radius of convergence
$\rho_3\geq \min\{\rho_1,\rho_2\}$.  

\item[$\circ$] Given $f(z)=\sum a_n(z-z_0)^n$ in the disk of convergence $D(z_0;\rho)$, with $\rho>0$, then there exists $f'(z)=\sum na_n(z-z_0)^{n-1}$ for all $z \in D(z_0;\rho)$. Thus, $f$ is infinitely differentiable in $D(z_0;\rho)$ and we have $a_n=\frac{ f^{(n)}(z_0)}{n!}$, for all $n \in \mathbb N$. We say that $f(z)=\sum  f^{(n)}(z_0)(z-z_0)^n/n!$ is the \textit{Taylor series} of $f$ around (or centered at) $z_0$. If $w\in D(z_0;r)$, then the Taylor series of $f$ around $w$ converges to $f$ in the open disk $D(w;r-|w-z_0|)$. 

\item[$\circ$] If $f(z)=\sum a_nz^n$ and $g(z)=\sum b_nz^n$ are both
convergent in $D(0;r)$, with $r>0$ and $g(0)\in D(0;r)$, then the function composition $(f\circ g)(z)=f\big(g(z)\big)$ is a convergent power series in some disk $D(0;\delta)$, where $\delta >0$.  

\item[$\circ$] If $f(z)=\sum a_nz^n$ is a convergent power series in
$D(0;r)$, where $r>0$ and $a_0=f(0)\neq 0$, then the function $1/f(z)$ is a convergent power series in some disk $D(0;\delta)$, with $\delta >0$.  

\item[$\circ$] {\sf Principle of Isolated Zeros for Power Series.} If
$f(z)=\sum a_nz^n$ is a power series convergent inside $D(0;r)$,  where $r>0$, such that $f(0)=0$ but $f$ is not the null function, then there exists a smallest $k\geq 1$ satisfying $a_k\neq 0$ and a power series $g(z)=\sum b_nz^n$ convergent in $D(0;\delta)$, for some $\delta >0$, so that we have the factorization $f(z) =z^kg(z)$, for all $z \in D(0;\delta)$, with $g$ nowhere vanishing.  

\item[$\circ$]{\sf Identity Principle for Power Series.} If $f(z)=\sum
a_nz^n$ and $g(z)=\sum b_nz^n$  are convergent power series in $D(0;r)$
satisfying $f(z)=g(z)$ for all $z$ in a subset $X$ of $D(0;r)$, where $X$ has an accumulation point in $D(0;r)$, then we have $a_n=b_n$ for all $n\in \mathbb N$.  

\item[$\circ$] The complex series
$\exp(z)=e^z=\sum_{n=0}^{+\infty}\frac{z^n}{n!}$, $\sin z=
\sum_{n=0}^{+\infty} \frac{(-1)^nz^{2n+1}}{(2n+1)!}$, and $\cos
z=\sum_{n=0}^{+\infty}\frac{(-1)^nz^{2n}}{(2n)!}$ converge in $\mathbb C$. Moreover, we have {\sf Euler's Formula}: $e^{i\theta}=\cos \theta + i\sin\theta$, for all $\theta \in \mathbb R$. 

\item[$\circ$] If $\alpha$ is a real number we define, for each $n\in \mathbb N$, the binomial coefficients $\binom{\alpha}{n}= \frac{\alpha(\alpha -1)\ldots (\alpha -n +1)}{n!}$, if $n\geq 1$, and $\binom{\alpha}{0}=1$. Then, we have the real binomial series $(1+x)^\alpha =\sum_{n=0}^{+\infty}\binom{\alpha}{n}x^n,$ with radius of convergence $\rho=1$. 

\end{itemize}
Right below we prove a result about the \textit{complex binomial series} that we shall need.

\begin{proposition}\label{PROP21} Let $p\in \mathbb N\setminus\{0\}$. Then $B(z)= \sum_{n=0}^{+\infty}\binom{1/p}{n}z^n$ converges in the open disk $D(0;1)$ and $B(z)$ is a pth root of $1+z$, where $|z|<1$. That is, we have \[B(z)^p=1+z,\ \textrm{for all}\  z \in D(0;1).\] 
\end{proposition}
\begin{proof} Given $x\in (-1,1)$, it is well-known that
$b(x)=(1+x)^\frac{1}{p}=\sum_{n=0}^{+\infty}\binom{1/p}{n}x^n$. It is also known that the real series $\sum_{n=0}^{+\infty}\binom{1/p}{n}x^n$ diverges if $|x|>1$. Therefore, from the Cauchy-Hadamard formula it follows that the function $B(z)=\sum_{n=0}^{+\infty}\binom{1/p}{n}z^n$, where $z \in \mathbb C$, has radius of convergence $\rho=1$. By a property of the product of convergent power series, the function $B(z)^p$ is a power series convergent in $D(0;1)$ that satisfies the equation $B(x)^p=b(x)^p=1+x$, for all $x\in (-1,1)$. The claim then follows from the identity principle. 
\end{proof}

Henceforth, $\Omega$ denotes an open subset of $\mathbb C$. 

We say that $\Omega$ is \textit{connected} if the only subsets $X$ of $\Omega$ such that $X$ and $\Omega \setminus X$ are both open in $\mathbb C$, are the subsets $X=\Omega$ and $X=\emptyset$.
 
Given $z_1$ and $z_2$, both in $\mathbb C$, we denote the line segment joining them by $[z_1,z_2]$.  A polygonal line is a finite union of line segments of the form $[z_0,z_1]\cup [z_1,z_2]\cup\ldots \cup [z_{n-1},z_n]$.
We say that $\Omega$ is \textit{polygonally connected} if each pair of points in $\Omega$ can be joined through line segments lying in $\Omega$. It is not difficult to verify that $\Omega$ is connected if and only if $\Omega$ is polygonally connected.

\begin{definition}\label{DEF22} A function $f:\Omega \to \mathbb C$ is called \textit{analytic} in $\Omega$ if for each $z_0\in \Omega$ there exists a radius $r=r(z_0)>0$ and constants $c_n \in \mathbb C$ such that $D(z_0;r) \subset \Omega$ and $f(z)= \sum_{n=0}^{+\infty}c_n(z-z_0)^n$ for all $z \in D(z_0;r)$. We indicate by $\mathcal{A}(\Omega)$ the set of  analytic functions in $\Omega$. 
\end{definition}

We will use the following well-known results on analytic functions, all of them easy consequences of the previously listed basic results on power series: 
\begin{itemize}
\item[$\circ$] Every power series convergent in
$D(0;\rho)$, where $\rho>0$, is analytic in $D(0;\rho)$. 
\item[$\circ$] Every analytic function is continuous and infinitely differentiable.
\item[$\circ$] If $f$ and $g$ are in $\mathcal{A}(\Omega)$ and $\lambda \in \mathbb C$, then $f+g$, $\lambda f$, and $fg$ are also in
$\mathcal{A}(\Omega)$. The function $1/f$ defined on the open set $\{z\in \Omega: f(z)\neq 0\}$ is also analytic.  
\item[$\circ$] If $g\in \mathcal{A}(\Omega_1)$ and $f\in
\mathcal{A}(\Omega_2)$ and the image of $g$, the set $g(\Omega_1)$, is
contained in $\Omega_2$, then the function composition $(f\circ
g)(z)=f\big(g(z)\big)$ is analytic in $\Omega_1$. 
\end{itemize}

Next, we prove a fundamental result about analytic functions.

\begin{proposition}\label{PROP23} ({\sf Principle of Isolated Zeros for
$\mathcal{A}(\Omega)$}) Let $f$ be in $\mathcal{A}(\Omega)$, with $f$ not being the null function, and $\Omega$ an open connected set in the complex plane. Then, $\mathcal{Z}(f)=\{z \in \Omega:
f(z)=0\}$ is an isolated subset of $\Omega$. Moreover, if $z_0\in \mathcal{Z}(f)$, then there exists a smallest $k\geq 1$ and a function $\varphi\in \mathcal{A}(\Omega)$ such that we have the factorization
\[f(z) = (z-z_0)^k\varphi(z), \ \textrm{for all}\ z \in \Omega,\ 
\textrm{where}\ \varphi(z_0)\neq 0 .\]  
\end{proposition}
\begin{proof} First, let us show that 
\[W=\{w\in \Omega: f\ \textrm{is identically zero over some open disk centered at}\ w\}\] 
is an empty set. Clearly, $W$ is open. Moreover, let us suppose that $(w_n)$ is a sequence in $W$ such that $w_n\to \zeta$, where $\zeta\in \Omega$, as $n\to +\infty$. By the hypothesis on $(w_n)$ we have $f^{(j)}(w_n)=0$, for all $j\in \mathbb N$ and all $n\in \mathbb N$. Since $f$ and its derivatives are continuous, we obtain $f^{(j)}(\zeta)=0$, for all $j\in \mathbb N$. Developing $f$ by its Taylor series centered at $\zeta$ we deduce that $\zeta\in W$. Therefore, $\Omega \setminus W$ is also open. Thus, since $\Omega$ is connected and $f$ is not the null function, $W$ is empty.

By the previous paragraph, given $z_0 \in \mathcal{Z}(f)$, the Taylor series of $f$ centered at $z_0$ does not vanish identically. Employing the principle of isolated zeros for power series, we find $k\in \{1,2,3,\ldots\}$ and a small $r>0$ such that  
\begin{displaymath}
\left\{\begin{array}{ll}
f(z)=(z-z_0)^kg(z),  &\textrm{for all}\ z \in D(z_0;r),\\
g(z) =\sum\limits_{p= k}^{+\infty}a_p(z-z_0)^{p-k}, &\textrm{with}\ 
a_p=\frac{f^{(p)}(z_0)}{p\,!}, \ \textrm{for all} \  p\geq k, \ \textrm{and}\ g(z_0)=a_k\neq 0. 
\end{array}
\right.
\end{displaymath}
We complete the proof by defining
\begin{displaymath}
\varphi(z)  = \left\{\begin{array}{ll}
g(z),  &\textrm{if}\ z \in D(z_0;r)\\
\huge{\frac{f(z)}{(z-z_0)^k}}, &\textrm{if}\ z \in \Omega \setminus \{z_0\}.
\end{array}
\right.
\end{displaymath}
\end{proof}

\

\section{The Gutzmer-Parseval Inequality for Polynomials and for Analytic Functions, Cauchy's Inequalities, and Liouville's Theorem}

In this section we prove The Gutzmer-Parseval Inequality for Polynomials and, as a consequence, The Gutzmer-Parseval Inequality for Analytic Functions, Cauchy's Inequalities, the Maximum Modulus Principle, the Uniqueness Theorem for the Coefficients of a Power Series, and two Liouville's theorems.

In 1832, A. L. Cauchy already knew the inequalities bearing his
name. In 1888, A. Gutzmer published the formula 
\[\sum |a_n|^2r^{2n}=\frac{1}{2\pi}\int_0^{2\pi}|f(z_0
+re^{i\theta})|^2\,d\theta\,,\]
 for functions complex differentiable in an open set (holomorphic functions).  
 
Searching for an integration-free theory of holomorphic functions, P. Porcelli and L. M. Weiner ~\cite{PW}, in 1957, published the following Cauchy inequality for polynomials: `` If a polynomial  $P(z)=a_0+a_1z+ \cdots +a_nz^n$ satisfies $|P(z)|\leq M$, for all $|z|\leq R$, then we have $|a_j|\leq M/R^{j}$, for $j=0, \ldots,n$.'' This Cauchy inequality was applied in ~\cite{COPO1}. Another proof of this inequality, with an application, was given in ~\cite{COPO2}. See also Leland ~\cite{LE}. 

The reader is invited to look Weierstrass's nice proof of Cauchy's inequality for analytic functions (given in 1841) that is offered in Remmert ~\cite[p. 247]{RR1}. I had the luck of receiving this same invitation from R. B. Burckel and Paulo A. Martin.

\begin{lemma}\label{LEM31}({\sf The Gutzmer-Parseval Inequality for Polynomials}) Let $P(z) =\sum_{j=0}^na_jz^j$, with $n\geq 1$, be a polynomial and $r>0$. Let us define
\[ m(r)=\min\limits_{|z|=r}|P(z)| \ \textrm{and}\  M(r)=\max\limits_{|z|=r}|P(z)|.\]
Then, we have 
\[ m(r)^2 \leq \sum_{j=0}^n|a_j|^{\,2}\,|r|^{2j} \leq M(r)^2.\]
\end{lemma}
\begin{proof} Let us consider the number $\omega =e^{i\pi/n}$ (thus, $\omega^n=-1$) and the $2n$ polynomials $P_k(z)
=P(\omega^kz )$, where $0\leq k \leq 2n-1$. A short computation reveals that
\begin{equation}\label{EQ3.1.1}
 |P_k(z)|^2  = \sum_{0\leq j\leq
n}|a_j|^{\,2}\,|z|^{2j}  + \, 2\sum_{0\,\leq\mu< \nu\,\leq 
n}\textrm{Re}\left[\overline{a_{\mu}}\,\overline{z}^{\,\mu}a_{\nu}z^{\,\nu}
\overline{\omega}^{\,k\mu}\omega^{k\nu}\right],
\end{equation} 
and for $\mu<\nu$, the difference $\nu -\mu$ runs over $ \{1,\ldots ,n\}$. Writing $\overline{\omega}^{\,k\mu}\omega^{\,k\nu}=\omega^{k(\nu -\mu)}$ we obtain the finite geometric sum
\[\sum\limits_{k=0}^{2n-1}\omega^{k(\nu -\mu)}=  
\frac{1 -\omega^{2n(\nu -\mu)}}{1 - \omega^{\nu -\mu}} = 0.\]
Hence, for $0\leq \mu <\nu\leq n$, it follows that
$\sum_{k=0}^{2n-1}\left[\overline{a_{\mu}}\,\overline{z}^{\,\mu}a_{\nu}z^{\,\nu}   \overline{\omega}^{\,k\mu}\omega^{\,k\nu}\right] = 0$. Thus, employing (\ref{EQ3.1.1}) and these identities we arrive at 
\begin{equation}\label{EQ3.1.2}
 \sum_{k=0}^{2n-1}|P_k(z)|^2 = 2n\sum_{j=0}^n|a_j|^{\,2}\,|z|^{2j} .
\end{equation}
Now, since  $\min\limits_{|z|=r}|P_k(z)|=\min\limits_{|z|=r}|P(z)|$ and  
$\max\limits_{|z|=r}|P_k(z)|=\max\limits_{|z|=r}|P(z)|$, we deduce that
\[2nm(r)^2 \leq \sum_{k=0}^{2n-1}|P_k(z)|^2 \leq 2n M(r)^2, \ \textrm{if}\ |z|=r .\]
The claimed inequalities follow from these and (\ref{EQ3.1.2}).
\end{proof}
\begin{remark}\label{REM32} It is rather trivial to produce a proof of Lemma \ref{LEM31} that does not employ polynomial continuity. To do so, it is enough to replace $m(r)$ and $M(r)$ by $\inf\{|P(z)|: |z|=r\}$ and $\sup\{|P(z)|:|z|=r\}$, respectively. 
\end{remark}

\begin{theorem}\label{TEO33} ({\sf The Gutzmer-Parseval Inequality for Analytic Functions)}
Let $f(z)=\sum a_nz^n$ be a convergent power series in $D(0;R)$, where $R>0$. Given $r$ such that $0\leq r<R$, we have  
\[\sum|a_n|^{\,2}r^{2n}\leq M(r)^2, \ \textrm{where}\ M(r)=\max_{|z|=r}|f(z)| .\]
\end{theorem}
\begin{proof} Let $z$ be arbitrary in $\mathbb C$, with $|z|=r$. From the triangle inequality follows that
\[\Big|\sum_{n=0}^Na_nz^n\Big| \leq M(r)\, +\,
\Big|\sum_{n=N+1}^{+\infty}a_nz^n\Big| \leq M(r)\, +\,
\sum_{n=N+1}^{+\infty}|a_n|r^n,\ \textrm{for all} \ N\in \mathbb N.\]   
Thus, by the Gutzmer-Parseval inequality for polynomials (Lemma \ref{LEM31}) we obtain
\[\sum_{n=0}^N|a_n|^2r^{2n} \leq \left[\,M(r)\,+\,
\sum_{n=N+1}^{+\infty}|a_n|r^n\,\right]^2,\ \textrm{for all} \ N \in \mathbb N.\] 
Passing the last inequality to the limit as $N\to +\infty$, the claimed inequality follows.
\end{proof}

\begin{remark}\label{REM34} By redefining $M(r)$ as $\sup\{|f(z)|:|z|=r\}$, we obtain a proof of Theorem \ref{TEO33} that does not use function continuity. See also Remark \ref{REM32}.
\end{remark}

\begin{remark}\label{REM35} From Remark \ref{REM34} follows a proof of {\sf The Uniqueness Theorem for the Coefficients of a Power Series} (see ~\cite[pp.~112--113]{BE}) that does not employ function continuity. In fact, let us suppose that $\sum a_nz^n$ and $\sum b_nz^n$ satisfy $\sum a_nz^n=\sum b_nz^n$, for all $|z| <R$. Hence, $\sum (a_n-b_n)z^n$ vanishes everywhere in $D(0;R)$. By Remark \ref{REM34}, we obtain $ \sum|a_n-b_n|^2r^{2n}\leq 0$ for all $r$ such that $0\leq r<R$. Thus, we have $a_n=b_n$ for all $n \in \mathbb N$. 
\end{remark}

\begin{corollary}\label{COR36}({\sf Cauchy's Inequalities})
Keeping the theorem's notation, we have  
\[ |a_n|\leq \frac{M(r)}{r^n}\,,\ \textrm{for all}\  n \in \mathbb N.\] 
\end{corollary}
\begin{proof} It is straightforward from Theorem \ref{TEO33}.
\end{proof}

\begin{corollary}\label{COR37} Let us keep the hypothesis and the notation in Theorem \ref{TEO33}. If
$\sup\{|f(z)|: z\in D(0;R)\}\leq M$ and $0<r<R$, then we have 
\[\max\limits_{\overline{D}(0;\,r)}|f'|\leq \frac{M}{R -r} .\]  
\end{corollary}
\begin{proof} The Taylor series of $f$ centered at an arbitrary $z_0$ in
$\overline{D}(0;r)$ is   
\[f(z) = \sum \frac{f^{(n)}(z_0)}{n!}(z-z_0)^n, \ \textrm{if}\  |z-z_0|<R -r.\]
Hence, from Theorem \ref{TEO33} follows 
\[|f'(z_0)|\,|z-z_0|\leq M, \ \textrm{if}\  |z-z_0|<R -r .\]
Taking $|z-z_0|$, with $|z-z_0|<R-r$, arbitrarily near $R-r$, at the limit we find that $|f'(z_0)|(R -r)\leq M $.
\end{proof}

\begin{theorem}\label{TEO38} ({\sf Maximum Modulus Principle}) Let $f:\Omega \to \mathbb C$, where $\Omega$ is open and connected, be analytic. If $|f|$ has a local maximum, then $f$ is a constant. 
\end{theorem}
\begin{proof} Let $z_0$ be a point of local maximum of $|f|$. We can
clearly assume that $z_0=0$. Expressing $f$ by its Taylor
series around $z_0=0$, we write $f(z) =\sum a_nz^n$ with $z$ in a disk
$D(0;\delta)$, where $\delta >0$. By hypothesis, there is $r$, where $0<r<\delta$,
such that  
\[|\sum a_nz^n|\leq |f(0)|=|a_0|, \ \textrm{for all} \ z \in \overline{D}(0;r) .\]
Hence, from Theorem \ref{TEO33} follows
\[|a_0|^2 + \sum_{n=1}^{+\infty} |a_n|^2\,r^{2n} \leq|a_0|^2 .\]
Thus, we obtain $a_n=0$ if $n\geq 1$ and $f(z)=a_0$ for all $z\in D(0;r)$. By Proposition \ref{PROP23} (principle of isolated zeros), $f$ is a constant. 
\end{proof}

\begin{remark}\label{REM39} The strictly algebraic result asserting that a complex polynomial of degree $n\geq 1$ has at most $n$ zeros is well-known. Consequently, adapting the proof of Theorem \ref{TEO38} (maximum modulus principle) to the case where $f$ is a polynomial, in a very obvious way, it is rather easy to see that by employing Lemma \ref{LEM31} (Gutzmer-Parseval inequality for polynomials) and Remark 3.2, one can produce a trivial proof of the {\sf Maximum Modulus Principle for Polynomials} that does not require polynomial continuity. 
\end{remark}

\begin{remark}\label{REM310} The maximum modulus principle (Theorem \ref{TEO38}) yields a proof of the fundamental theorem of algebra that is very much like the classical proof (that employs Liouville's Theorem for holomorphic functions, see Theorem \ref{TEO312}) usually given in regular courses. In fact, suppose that there exists a complex polynomial $p=p(z)$, degree $(p)\geq 1$, with no zeros in $\mathbb C$. Therefore, $1/p$ is analytic in $\mathbb C$. Since $|p(z)|\to +\infty$ as $|z|\to +\infty$ (see \cite{OO1}), then $1/|p(z)|\to 0$ as $|z|\to +\infty$. Thus, $1/|p|$ has a global maximum at some $z_0\in \mathbb C$. From Theorem \ref{TEO38} it may be concluded that $1/p$ is a constant. Hence, degree$(p)=0$; which is a contradiction. 
\end{remark}
 
\begin{definition}\label{DEF311} An analytic function $f$ is \textit{entire} if its domain is $\mathbb C$ [i.e., $f\in \mathcal{A}(\mathbb C)$]. 
\end{definition}

Next, in Theorem \ref{TEO312} we prove Liouville's Theorem for an entire function $f$ supposing that $f$ is given by its Taylor series at the origin. We demonstrate on Theorem \ref{TEO91} (a result due to Hurwitz, Read, and  Connell-Porcelli) that such expansion occurs for every entire function.

The following proof of Theorem \ref{TEO312} does not employ function continuity.

\begin{theorem}\label{TEO312}({\sf Liouville}) Let $f(z)=\sum a_nz^n$ be
bounded and convergent in $\mathbb C$. Then, $f$ is a constant. 
\end{theorem}
\begin{proof} Let $M\in \mathbb R$ be such that $|f(z)|\leq M$, for all $z\in \mathbb C$. Employing Remark \ref{REM34} and a calculation similar to the one in the proof of Theorem \ref{TEO33}, we find (without using function continuity) the inequality 
\[\sum |a_n|^2r^{2n}\leq M^2, \ \textrm{for all}\ r\geq 0.\]
Hence, for a fixed $n\geq 1$ we obtain $|a_n|^2r^{2n}\leq M^2$, for all $r\geq 0$, which implies that $a_n=0$. As a result, we conclude the identity $f(z)=a_0$ for all $z\in \mathbb C$.
\end{proof}

\begin{remark}\label{REM313} Similarly to Theorem \ref{TEO312}, we have {\sf The Extended Liouville Theorem:} {\sl Let us suppose that $f(z)=\sum a_nz^n$ converges in the complex plane. If there are constants $A\geq 0$, $B\geq 0$, and  $N\in \mathbb N$ such that 
\[|f(z)| \leq A + B|z|^N,\ \textrm{for all}\ z \in \mathbb C\,,\ \]
then $f$ is a polynomial and degree$(f)\leq N$.} The proof is easy. In fact, since we have
$|\sum a_nz^n|\leq A+Br^N$, for all $|z|=r$, by employing Theorem \ref{TEO33} (the Gutzmer-Parseval inequaliy for analytic functions) and a
straightforward inequality, we conclude that 
\[\sum|a_n|^2r^{2n} \leq (A+Br^N)^2 \leq  C+Dr^{2N}, \ \textrm{for all}\ r\geq
0, \ \textrm{with}\ C\geq 0\ \textrm{and}\ D\geq 0\ \textrm{ constants}.\]
Hence, we have $a_n=0$ for all $n>N$. Thus, $f$ is a polynomial and degree$(f)\leq N$. 
\end{remark}

\

\section{The Maximum and Minimum Modulus Principles }

In spite of the Maximum Modulus Principle for analytic functions (Theorem \ref{TEO38}) being an easy consequence of the Gutzmer-Parseval Inequality for Analytic Functions (Theorem \ref{TEO33}), we present another proof of such principle in Theorem \ref{TEO41} on account of the following four reasons: (1) it furnishes analogous, independent, and easy direct proofs of the Maximum and Minimum Modulus Principles for analytic functions and also for polynomials (the proofs for polynomials can be freed from the exponential function by adapting the proof of the fundamental theorem of algebra in \cite{OO1}); (2) although the Minimum Modulus Principle for $f\in\mathcal{A}(\Omega)$ follows easily from the Maximum Modulus Principle for $f\in \mathcal{A}(\Omega)$ by applying the latter to $1/f$ if the function $f$ does not vanish, such an argument cannot be used if we are restricted to the algebra of polynomials because the function $1/P$ is not a polynomial if $P$ is a non-constant polynomial; (3) the importance of the Minimum Modulus Principle for polynomials in proving the fundamental theorem of algebra (see  \cite{OO1} and Remmert  \cite[p. 112]{RR2}); (4) its simplicity and usefulness in analyzing saddle points of $|f|$, where $f$ is an analytic function, and in proving the Anti-Calculus Proposition (Theorem \ref{TEO52}) and the Clunie-Jack Lemma (Theorem \ref{TEO102}). 

\begin{theorem}\label{TEO41} Let $f:\Omega \to \mathbb C$, where $\Omega$ is an open and connected  set in $\mathbb
C$, be analytic and non-constant. Then, 
\begin{itemize}
\item[(a)] ({\sf Maximum Modulus Principle}) $|f|$ has no local maximum.
\item[(b)] ({\sf Minimum Modulus Principle}) $|f|$ has no local minimum at
$z_0$ in $\Omega$, unless $f(z_0)=0$. 
\end{itemize}
\end{theorem}
\begin{proof} Let us suppose that $z_0=0$ is a point of local maximum or local minimum of $|f|$. Hence, the expression 
\begin{equation}\label{4.1.1}
 |f(z)|^2 -|f(0)|^2
\end{equation}
does not change sign [$\geq 0$ or $\leq 0$] in some disk $D(0;\epsilon)$,
with $\epsilon >0$, and by the principle of isolated zeros (Proposition \ref{PROP23}) there exists a natural number $k\geq 1$ and a function $\varphi$ analytic in $D(0;\epsilon)$ such that 
\begin{equation}\label{4.1.2}
 f(z) = f(0)+ z^k\varphi(z),\   \varphi(0)\neq  0.
\end{equation}
Putting $z=re^{i\theta}$, with $0<r<\epsilon$ and $\theta\in \mathbb R$, and
combining the expressions (\ref{4.1.1}) and (\ref{4.1.2}) we obtain 
\begin{equation}\label{4.1.3}
|f(re^{i\theta})|^2 -|f(0)|^2=2r^k\textrm{Re}\left[\overline{f(0)}e^{ik\theta}
\varphi(re^{i\theta})\right] + r^{2k}|\varphi(re^{i\theta})|^2  ,
\end{equation} 
whose sign is constant and unchanged when we divide the second member of
(\ref{4.1.3}) by $r^k$ (with $0<r<\epsilon$):   
\[2\textrm{Re}\big[\overline{f(0)}\varphi(re^{i\theta})e^{ik\theta}\big]
+ r^{k}|\varphi(re^{i\theta})|^2 .  \]  
Fixing $\theta \in \mathbb R$, by continuity the limit of the expression right above for $r\to 0^{+}$ is the expression 
\[2\textrm{Re}\big[\,\overline{f(0)}\varphi(0)e^{ik\theta}\,\big] \]  
that keeps the sign of the former expression, independently of $\theta\in
\mathbb R$. However, this is only possible if $\overline{f(0)}\varphi(0)=0$
[to see this, it is enough to choose values of $\theta$ such that
$e^{ik\theta}$ assumes the values $-1$, $+1$, $-i$, and $+i$]. Therefore, we deduce that $\overline{f(0)}\varphi(0)=0$ and then $f(0)=0$.

Consequently, if $z_0$ is a point of local maximum of $|f|$, then $f$ vanishes everywhere in a neighbourhood of $z_0$ and through the principle of isolated zeros we conclude that $f$ is null in $\Omega$, against the hypothesis. Thus, there is no such $z_0$. It is quite easy to turn this proof of the maximum modulus principle into a direct proof. 

If $z_0$ is a point of local minimum, then we have proved $f(z_0)=0$ as desired.
\end{proof}

Next, we give a very intuitive proof of Liouville's Theorem for a bounded power series convergent in  the entire complex plane, employing the maximum modulus principle [Theorem \ref{TEO41} (a)]. See also Theorem \ref{TEO312} and Remark 3.13.

\begin{theorem}\label{TEO42} ({\sf{Liouville}}) Let $f(z)=\sum a_nz^n$, where $z\in \mathbb C$, be a 
bounded function. Then, $f$ is constant. 
\end{theorem}
\begin{proof} Clearly, the function $|f(z)-a_0|= |z||a_1+a_2z +a_3z^2+ \cdots |$ is bounded. Hence, if $|z|\to +\infty$, then $\varphi(z)=a_1+a_2z +a_3z^2+ \cdots $ tends to $0$ and the function $|\varphi(z)|$ has a global maximum. By the maximum modulus principle [Theorem \ref{TEO41} (a)], we see that $\varphi$ is constant. Thus, $\varphi$ is the zero function and $f$ is constant.
\end{proof}

\begin{remark} One can adapt the proof of Theorem \ref{TEO42} for an entire and bounded analytic function $f$ [i.e., $f$ is bounded and $f\in \mathcal{A}(\mathbb C)$]. In fact, applying the principle of isolated zeros (Proposition \ref{PROP23}) and writing $f(z)=a_0+z\varphi(z)$, with $\varphi$ analytic in $\mathbb C$, one can then proceed exactly as in the proof of Theorem \ref{TEO42}. 
\end{remark}

In Theorem \ref{TEO44} we present a proof of the maximum modulus principle for an analytic function that does not employ function continuity.

\begin{theorem}\label{TEO44} ({\sf Maximum Modulus Principle}) Let $f$ be analytic in an open and connected set $\Omega$. If the function $|f|$ has a local maximum at some $z_0\in \Omega$, then $f$ is a constant.
\end{theorem}
\begin{proof}
At first, let us first consider the case $f(z)=\sum a_nz^n$, with $z \in D(0;R)$ and $R>0$. Then, let us define $g(w)=f(z_0+w)$, with $|w|<R-|z_0|$. Since $\sum a_nz^n$ converges absolutely within $D(0;R)$, by a usual power series computation we have  
\begin{displaymath}
\left\{\begin{array}{ll}
g(w)=\sum\limits_{n=0}^{+\infty} a_n(z_0+w)^n= \sum\limits_{n=0}^{+\infty}\sum\limits_{m=0}^n\binom{n}{m}a_nz_0^{n-m}w^m= \sum\limits_{m=0}^{+\infty}b_mw^m,\\
\\
\textrm{where}\ b_m=\sum\limits_{n\geq m}a_n\binom{n}{m}z_0^{n-m}.
\end{array}
\right.
\end{displaymath}
Moreover, since $|g|$ has a local maximum at $w=0$, from Remark \ref{REM34} it follows (without employing function continuity) the inequality $|b_0|^2 + \sum_{m=1}^{+\infty} |b_m|^2\rho^{2m} \leq |b_0|^2$, for all $\rho >0$ and $\rho$ small enough. Hence, we obtain $b_m=0$, for all $ m\geq 1$, which implies that $g(w)=g(0)=f(z_0)$, for all $w \in D(0;R - |z_0|)$.

Thus, we proved that $f(z)=f(z_0)$, for all $z\in D(z_0;R-|z_0|)$ (the biggest open disk centered at $z_0$ and still inside the domain of $f$). Let us consider the subcase $0 \in D(z_0;R-|z_0|)$. Then, $0$ is a point of local maximum of $|f|$. Hence, by the previous argument we deduce that $f(z)=f(0)$, for all $ z \in D(0;R)$. The proof of this subcase is complete.

Let us now consider the subcase $0 \notin D(z_0;R-|z_0|)$. Then, the point $z_1= z_0 - \frac{R-|z_0|}{2}\frac{z_0}{|z_0|} \in D(z_0;R-|z_0|)$ is a point of local maximum of $|f|$. Hence, by the argument in the paragraph right above, it follows that $f(z)=f(z_1)=f(z_0)$, for all $z \in D(z_1; 3(R-|z_0|)/2)$, noticing that $3(R-|z_0|)/2=R-|z_1|$. Otherwise, if the point $0$ is in $D(z_1; R-|z_1|)$, then $0$ is a point of local maximum of $|f|$. Thus, we complete the proof of this subcase as in the paragraph right above. 

If $0\notin D(z_1;R-|z_1|)$, then proceeding as in the last paragraph we will eventually find a point $z_n= z_{n-1}- \frac{R -|z_{n-1}|}{2}\frac{z_{n-1}}{|z_{n-1}|}\in D(z_{n-1};R-|z_{n-1}|)$, with $n\geq 2$ and radius $R-|z_{n-1}|=(\frac{3}{2})^{n-1}(R-|z_0|)$, such that $f(z)=f(z_n)= \cdots =f(z_1)=f(z_0)$, for all $z \in D(z_n;R -|z_n|)$, with $0\in D(z_n;R-|z_n|)$. Hence, in such a subcase, we complete the proof as in the last paragraph. The proof of the first case is complete.

Now, let us consider the general case $f:\Omega \to \mathbb C$. Then, since $\Omega$ is open and connected, given any $w\in \Omega$, there exists a polygonal path $P=[z_0,z_{1}]\bigcup \ldots \bigcup[z_{n-1},z_{n}]$, where $n$ is fixed and $[z_j,z_{j+1}]=\{z_j+t(z_{j+1}-z_j):0\leq t\leq1\}$, for all $j\in \{0,\ldots, n-1\}$, inside $\Omega$ and connecting $z_0$ and $z_n=w$. Moreover, for every $\zeta\in P$, there exists $r_\zeta>0$ such that $f$ is developable as a power series centered at $\zeta$ and convergent in $D(\zeta;r_\zeta)$. Furthermore, since the polygonal $P$ is compact, there exists a smallest $m\geq 0$ such that $\zeta_0=z_0,\ldots,  \zeta_m$ satisfies $P\subset D(\zeta_0;r_{\zeta_0}) \bigcup \ldots \bigcup D(\zeta_m;r_{\zeta_m})$. Then, since the intersections $P\cap D(\zeta_j;r_{\zeta_j})$ are nonempty, for all $ j\in\{0,\ldots, m\}$, we can re-enumerate these disks so that for each $j \in\{1,\ldots, m\}$, the disk $D(\zeta_j;r_{\zeta_j})$ intersects the union $D(\zeta_0;r_{\zeta_0})\cup \ldots\cup D(\zeta_{j-1};r_{\zeta_{j-1}})$. Now, by the first case, we notice that we have $f(z)=f(z_0)$, for all $ z \in D(z_0;r_{z_0})$. Consequently, we obtain $f(z)=f(z_0)$, for all $z$ in  $D(\zeta_0;r_{\zeta_0})\cap D(\zeta_1;r_{\zeta_1})$. Therefore, by the first case, we conclude that $f(z)=f(z_0)$, for all $z \in D(\zeta_1;r_{\zeta_{1}})$, which implies the identity $f(z)=f(z_0)$, for all $z \in D(z_0;r_{z_0}) \cup D(\zeta_{1};r_{\zeta_{1}})$. Hence, proceeding by induction we conclude that $f(z)=f(z_0)$, for all $z \in D(\zeta_0;r_{\zeta_0})\cup \ldots \cup D(\zeta_m;r_{\zeta_m})$. The proof is complete.
\end{proof}

\

\section{The Anti-Calculus Proposition by P\'{o}lya-Szeg\"{o} and P. Erd\"{o}s and the Theorem on Saddle Points by Bak-Ding-Newman.}

The two results in this section can be found in Bak and Newman ~\cite[pp. ~87--90]{BN} and \cite{BDN}, and also in H. P. Boas \cite{BO}. As proposed problems, they can be seen in P\'{o}lya and Szeg\"{o} ~\cite[Part III, Problems 132, 136, and 144]{POSZ}. The proofs given in this section are different than the respective ones in [1], [2], and [4].

\begin{definition}\label{DEF51} Given a set $X\subset \mathbb C$ and a function $f:X\to \mathbb C$, we say that $f$ is \textit{ analytic in $X$} if there exists an open set containing $X$ in which some extension of $f$ is analytic.
\end{definition}

\begin{theorem}\label{TEO52} ({\sf Anti-Calculus Proposition (P\'{o}lya-Szeg\"{o}, P. Erd\"{o}s)}) Let us suppose that
$f:\overline{D}(0;R)\to \mathbb C$ is analytic and non-constant and $\alpha$, where $|\alpha|=R$, is a point of maximum of $|f|$ or a point of minimum of $|f|$. 
\begin{itemize}
\item[(a)] If $\alpha$ is a point of maximum, then $f'(\alpha)\neq 0$. 
\item[(b)] If $\alpha$ is a point of minimum, then $f(\alpha)=0$ or
$f'(\alpha)\neq 0$. 
\end{itemize}
\end{theorem}  

\begin{proof} Let us suppose that $f'(\alpha)=0$. Since $\alpha$ is either a point of maximum or a point of minimum of $|f|$ in $\overline{D}(0;R)$ and $f$ is
non-constant, from the Taylor series of $f$ centered at $\alpha$ we may conclude that there exists a function $\varphi$ analytic inside $D(\alpha;r)$, with $r>0$ and sufficiently small, and $k\geq 2$ satisfying
\begin{displaymath}
(5.1) \ \ \ \ \ \  \ \left\{\begin{array}{ll}
f(z) = f(\alpha) + (z-\alpha)^k\varphi(z), \ \ \varphi(\alpha)\neq 0,\ \ 
\ \textrm{for all}\ z \in D(\alpha;r), \\ 
\\
|f(z)|^2 - |f(\alpha)|^2 \,\ \ \    \textrm{has the same sign for all} \
z \in D(\alpha;r)\cap D(0;R) . 
\end{array}
\right.\ \ \ \ \ \ \ \ \ \ \ \ \ \ \ \ \ 
\end{displaymath}
Substituting the first equation of (5.1) for $z=\alpha +\epsilon e^{i\theta}$, where the real numbers $\epsilon$ and $\theta $ are such that $z\in D(\alpha;r)\cap D(0;R)$, into the second expression of (5.1) we obtain the expression 
\[(5.2)\ \ \ \ \ \  \ \ \ \ \ \ \ \ \ 
2\epsilon^k\textrm{Re}\big[\,\overline{f(\alpha)}e^{ik\theta}\varphi(\alpha
+\epsilon e^{i\theta})\,\big] + \epsilon^{2k}\,|\varphi(\alpha + \epsilon
e^{i\theta})|^2,\ \ \ \ \ \ \ \ \ \ \ \ \ \ \ \ \ \ \ \ \ \ \ \ \ \ \ \ \]
whose sign is the same for all such values of  $\epsilon$ and $\theta$. The possible values of $\theta$ include some interval $(\theta_0, \theta_0 + \pi)$. Given one such $\theta$, there
exists $r_\theta>0$ such that $z=\alpha +\epsilon e^{i\theta} \in D(0;R)$ if $\epsilon$ is in $(0,r_\theta)$. Let us fix one such
$\theta$. Taking the limit as $ \epsilon \to 0^{+}$ of (5.2) divided by
$\epsilon^k$ we obtain the expression 
\[ 2\textrm{Re}\big[\,\overline{f(\alpha)}\varphi(\alpha)e^{ik\theta}\,\big]
,\ \ \textrm{with the same sign for all}\    \theta \in (\theta_0,\theta_0
+\pi).\] 
However, since $k\geq 2$, this can only be true if
$\overline{f(\alpha)}\varphi(\alpha)=0$, which implies $f(\alpha)=0$.  

Now we can complete the proof of this theorem.

\begin{itemize}
\item[(a)] If $\alpha$ is a point of maximum, then we deduce that $f$ is the zero function, against the hypothesis. Thus, $f'(\alpha)\neq 0$. 
\item[(b)] If $\alpha$ is a point of minimum then we proved $f(\alpha)=0$, as we intended.
\end{itemize}
\end{proof}

\begin{definition}\label{DEF53} Given a real-differentiable function $F:\Omega\to \mathbb R$, with $\Omega$ an open subset of $\mathbb R^2$, a point $P_0\in\Omega$ is a \textit{ saddle point} of $F$ if $P_0$ is a critical point of $F$ (i.e., the partial derivatives of first order $F_x(P_0)$ and $F_y(P_0)$ are both zero) but $P_0$ is not a local extremum of $F$ (i.e., $P_0$ is neither a point of local maximum of $F$ nor a point of local minimum of $F$). 
\end{definition}

Let us consider an analytic function $f:\Omega \to \mathbb C$, where $\Omega$ is an open subset of  $\mathbb C$. By identifying $\Omega$ as the subset $\{(x,y) \in\mathbb R^2: z= x+iy\in \Omega\}$ contained in $\mathbb R^2$, we consider in $\mathbb R^3$ the graph of the function $|f|:\Omega \to \mathbb R$,   
\[\textrm{Graph}(|f|)=\Big\{\,(x\,,y\,,|f(z)|\,)\, \in \mathbb R^3 :\ z=x +iy\in \Omega \Big\}.\]

\begin{definition}\label{DEF54} The set Graph$(|f|)$ is the 
\textit{ analytic landscape} of $f$ (see Busam and Freitag \cite[p. 64]{BF}, Jensen \cite{JEN}, and  P\'{o}lya-Szeg\"{o}  \cite[Part III, Problems 131 and 132]{POSZ}).  
\end{definition}

\begin{theorem}\label{TEO55} ({\sf Bak-Ding-Newman}) Let $f$ be a non-constant analytic
function defined in $\Omega$ and $z_0=x_0+iy_0\in \Omega$, with $x_0$ and
$y_0$ real numbers. Then, $(x_0,y_0)$ is a saddle point of $F=|f|:\Omega \to \mathbb R$ if and only if 
$f'(z_0)=0$ and $f(z_0)\neq 0$. 
\end{theorem}
\begin{proof} Let us keep the notation in Theorem \ref{TEO52} (the anti-calculus proposition). We may assume without loss of generality that $z_0=0$.
\begin{itemize}
\item[($\Rightarrow$)] It is clear that $f(0)\neq 0$. Let us consider $\psi$
 analytic in $\Omega$ such that
\[f(z) = f(0)+ z\psi(z), \ \psi(0)=f'(0).\]
At the origin, $|f|$ is real-differentiable and has null gradient. Then,
fixing an arbitrary angle $\theta \in \mathbb R$ and taking the limit for
$r\to 0^{+}$ we obtain 
\[\ \ (5.3)\ \ \ 0  =\lim\limits_{r\to 0^{+}}\frac{|f(re^{i\theta})| -|f(0)|}{|re^{i\theta}|}\ \   \ \ \ \ \ \ \ \ \ \ \ \ \ \ \ \ \ \ \ \ \ \ \ \ \ \ \ \ \ \ \ \ \ \ \ \ \ \ \ \ \ \ \ \ \ \ \]
\[\ \ \ \ \ \ \ \ \ \ \ \ \ \ \ \ = \lim\limits_{r\to 0^{+}}\frac{1}{|f(re^{i\theta})|+|f(0)|}\,\Big\{2\textrm{Re}\left[\,\overline{f(0)}
\psi(re^{i\theta})e^{i\theta}\,\right] \ +\ r|\psi(re^{i\theta})|^2\Big\}.\]
On the other hand, for fixed $\theta$, if $r\to 0^{+}$ then we have 
$f(re^{i\theta})\to 0$ and, in addition, $\psi(re^{i\theta})\to f'(0)$. Substituting these
limits into (5.3) we obtain 
\[2\textrm{Re}[\overline{f(0)}f'(0)e^{i\theta}]=0 .\]
Since $\theta$ is arbitrary in $\mathbb R$, we finally obtain 
$\overline{f(0)}f'(0)=0$ and thus $f'(0)=0$. 
\item[($\Leftarrow$)] Thanks to the hypothesis, we know that there exists a function $\varphi$ analytic in $\Omega$ which satisfies the equation
$f(z)=f(0)+z^k\varphi(z)$, for all $z\in \Omega$, where $k\geq 2$ and $\varphi(0)\neq 0$. Hence, given $z\neq 0$ we can write 
\[\ \ \ \ \ \frac{|f(z)| -|f(0)|}{|z|}=\frac{|f(z)|^2
-|f(0)|^2}{|z|\,(\,|f(z)|+|f(0)|\,)}=
\frac{2\textrm{Re}[\,\overline{f(0)}\varphi(z)z^k\,]\,+\,
|\varphi(z)|^2\,|z|^{2k}}{|z|\,(\,|f(z)|+|f(0)|\,)} .\] 

Consequently, we have
\[\lim_{z \to 0}\frac{|f(z)|
-|f(0)|}{|z|}=\ \ \ \ \ \ \ \ \ \ \ \ \ \ \ \ \ \ \ \ \ \ \ \ \ \ \ \ \ \ \ \ \ \ \ \ \ \ \ \ \ \ \ \ \ \  \ \ \ \ \ \ \ \ \ \ \ \ \ \ \   \]
\[\ \ \ \ \ \ \ \ \ \ \ \ \ \ \ \ \ \ \ =\lim_{z \to 0}\frac{1}{|f(z)|+|f(0)|}\,\left\{2\textrm{Re}\left[\,
\overline{f(0)}\varphi(z)z^{k-1}\frac{z}{|z|}\,\right]
\ +\ |\varphi(z)|^2|z|^{2k-1}\right\}=0 .\] 
Therefore, at the origin the function $|f|$ is real-differentiable and its partial derivatives of first order vanish. 

Furthermore, for any fixed angle $\theta$, the sign of
$|f(re^{i\theta})|^2-|f(0)|^2$ is that of the expression
$2\textrm{Re}\big[\overline{f(0)}\varphi(re^{i\theta})e^{ik\theta}\big] \
+\ r^k|\varphi(re^{i\theta})|^2$ which is, considering small values of $r>0$, the sign of $2\textrm{Re}[\overline{f(0)}\varphi(0)e^{ik\theta}]$, if this
number is not zero. However, since $\overline{f(0)}\varphi(0)\neq 0$, this
sign can in fact change, depending on the chosen $\theta$. Thus, $z=0$ is a
saddle point. 
\end{itemize}
\end{proof}
\

\section{The Polygonal Mean-Value Property for Polynomials} 

The next result is just a part of a theorem due to S. Kakutani and M. Nagamo
\cite{KN} and J. L. Walsh \cite{WA} (see also S. Haruki \cite{SH}) characterizing the functions having the property described in Definition \ref{DEF61} as polynomials.

\begin{definition}\label{DEF61} A function $f(z)$ possesses the 
\textit{ polygonal mean-value property} if there exists $N\in \mathbb N$ such that for any $z_0 \in \mathbb C$ the value $f(z_0)$ is the average of $f$ at the $N$ vertices of every regular polygon with $N$ sides centered at $z_0$. 
\end{definition}

\begin{theorem}\label{TEO62} ({\sf Kakutani-Nagamo, Walsh}) Let $P(z)=\sum_{j=0}^na_jz^j$,
$n\geq 1$, be a complex polynomial and $\omega=e^{i\pi/n}$ (thus,
$\omega^n=-1$). Given arbitrary complex numbers $z_0$ and $z$, we have 
\[ P(z_0) =
\frac{1}{2n}\sum_{k=0}^{2n -1}P(z_0 + z\omega^k)  . \] 
\end{theorem}
\begin{proof} Supposing initially $z_0=0$, we consider the $2n$ polynomials \[P_k(z) = P(z\omega^k)=a_0 +a_1z\omega^k + \cdots +a_nz^n\omega^{kn},\ \textrm{where}\ 0\leq
k\leq 2n-1.\] 
Clearly, we have
\begin{displaymath}
\sum_{k=0}^{2n-1}\omega^{kj} = \left\{\begin{array}{ll}
2n,  & \textrm{if}\ j=0\\
\frac{1 -\omega^{2nj}}{1 -\omega^j}=0, & \textrm{if} \ 1\leq j\leq n.
\end{array}
\right.
\end{displaymath} 
Hence, $\sum_{k=0}^{2n-1}P_k(z) = 2na_0=2nP(0)$. Thus, defining
$Q(z)=P(z_0+z)$ we obtain
\[2nP(z_0)=2nQ(0)=\sum\limits_{k=0}^{2n-1}Q_k(z)=\sum\limits_{k=0}^{2n-1}Q(z\omega^k)
=\sum\limits_{k=0}^{2n-1}P(z_0+z\omega^k).\]  
\end{proof}

\begin{remark} Since the enunciation of Theorem \ref{TEO62} does not specify the degree of the polynomial $P(z)=\sum_{j=0}^na_jz^j$, we may conclude that its claim is true for every $n\in \{1,2,3,\ldots\}$ such that $n\geq$ degree$(P)$.
\end{remark}

As a consequence of Theorem \ref{TEO62}, we provide a ``high school proof'' of the Maximum Modulus Principle for Polynomials that does not employ polynomial continuity (see Remark \ref{REM39} and the penultimate paragraph of the Introduction).

\begin{theorem}\label{TEO64}({\sf Maximum Modulus Principle for Polynomials}) Let $P(z) =\sum_{j=0}^na_jz^j$, where $n\geq 1$, be a complex polynomial and $z_0$ a point of local maximum of $|P|$. Then, $P$ is a constant. 
\end{theorem}
\begin{proof} We assume, without losing generality, $z_0=0$. Let $R>0$ be such that $|P(z)|\leq |P(0)|$, for all $ z \in \overline{D}(0;R)$, and let
$\omega=e^{i\pi/n}$ (thus, $\omega^n=-1$). Then, by the polygonal mean-value
property for polynomials (Theorem \ref{TEO62}) we have
\[P(0)=\frac{1}{2n}\sum_{k=0}^{2n-1}P(z\omega^k)\]
and also, since $|P(0)|$ is the maximum value of $|P|$ in $\overline{D}(0;R)$,   
\[|P(0)| \leq  \frac{1}{2n}\Big[|P(z\omega^0)| +\cdots +
|P(z\omega^{2n-1})|\Big] \leq  \frac{1}{2n}\Big[|P(z)|\,+\,
(2n-1)|P(0)|\Big] \leq |P(0)| .\] 
Thus, we obtain $|P(z)|=|P(0)|=|a_0|$ for all $z \in \overline{D}(0;R)$. Hence, given a real number $z=x\in[0,R]$ we deduce that 
\[|a_0|^2= \sum_{j=0}^{n}|a_j|^2x^{2j}  \,+\, 2\sum_{0\leq \mu<\nu\leq
n}\textrm{Re}[\overline{a_\mu}a_\nu]x^{\mu +\nu}\] 
and, cancelling $|a_0|^2$ on each side, noticing that $\mu +\nu<2n$ if $0\leq \mu<\nu\leq n$, and isolating the monomial $|a_n|^2x^{2n}$, which has the biggest exponent, 
\[\left\{\sum\limits_{j=1}^{n-1}|a_j|^2x^{2j}  + 2\sum\limits_{0\leq \mu<\nu\leq
n}\textrm{Re}[\,\overline{a_\mu}a_\nu\,]x^{\mu +\nu}\right\} +  |a_n|^2x^{2n}=0, \
 \textrm{for all}\ x\in [0,R].\] 
Since every complex polynomial with a nonzero coefficient has a finite number of zeros, we deduce that all the coefficients of the polynomial right above are zero. Thus, the coefficient $|a_n|^2$ of the monomial $|a_n|^2x^{2n}$ is zero. Hence, we are allowed to write $P(z)=\sum_{j=0}^{n-1}a_jz^j$. If $n-1=0$, then we have $P(z)=a_0$; thus, the proof is complete. Otherwise,  by repeating the previous argument $(n-1)$-times we conclude that $a_{n-1}=a_{n-2}=\cdots =a_1=0$ and consequently
$P(z)$ is a constant.
\end{proof}

\newpage

\begin{proposition}\label{PROP65} ({\sf Cauchy's Inequalities for $P(z)
=\sum_{j=-n}^{j\,=\,m} a_jz^j$}) Let us consider $P(z) =\sum_{j=-n}^{j\,=\,m} a_jz^j$, where $n$ and $m$ are fixed in $\mathbb N$ and $a_j\in \mathbb C$, for all $j$ such that $-n\,\leq j\leq m$. Let $M(r)=\max\{|P(z)|: |z|=r\}$, where $r>0$. Then,
\[ |a_j|\leq \frac{M(r)}{r^j},   \ \textrm{for all}  \ j \in \{-n, -n+1, \ldots, m\}.\] 
\end{proposition}
\begin{proof} Introducing null coefficients $a_{j's}$ if necessary, we can
suppose that $m=n$. Set $\omega=e^{i\pi/n}$. A short computation reveals that
\[P_k(z) = P(z\omega^k)=
\frac{a_{-n}}{z^n}\omega^{-kn}+ \cdots +a_0+ \cdots+a_nz^n\omega^{kn} ,\ 
\textrm{if}\  0\leq k\leq 2n -1,\]  
and 
\begin{displaymath}
\sum\limits_{k=0}^{2n-1}\omega^{kj} = \left\{\begin{array}{ll}
2n, & \textrm{if} \ j=0\\
\frac{1 - \omega^{2nj}}{1 -\omega^j} = 0, & \textrm{if} \ -n\leq j\leq n\
\textrm{and}\ j\neq 0. 
\end{array}
\right.
\end{displaymath}
Hence, since $\max\{|P_k(z)|: |z|=r\}=\max\{|P(z)|: |z|=r\}=M(r)$, for 
$|z|=r$ we obtain
\[2n|a_0| = \left|\sum\limits_{k=0}^{2n-1}P_k(z)\right| \leq   2nM(r)
,\] 
and then for the constant term $a_0$ we deduce the inequality $|a_0|\leq M(r)$. 

For $j\in \mathbb N$ such that $-n\leq j\leq n$ we have $P(z)=z^jQ(z)$, with
$Q(z)=a_{-n}z^{-n-j}+ \cdots +a_j+ \cdots +a_nz^{n-j}$ and  $\max\{|Q(z)|:
|z|=r\}=M(r)/r^j$. Thus, by the previous case we conclude that
\[|a_j|\leq \max_{|z|=r}|Q(z)|=\frac{M(r)}{r^j}. \]
\end{proof}

The following result is a discrete version of Cauchy's Integral Formula.

\begin{proposition}\label{PROP66} Let $P(z)$ be a complex polynomial, $n \in \{1,2,3, \ldots\}$ such that $n\geq$ degree$(P)$, and $\omega =e^{i\pi/n}$ (thus, $\omega^n=-1$). Then, given arbitrary complex numbers $z_0$ and $z$, with $z\neq 0$, and an arbitrary $j\in \{0,1, \ldots, n\}$ we have  
\[\frac{P^{(j)}(z_0)}{j\,!} =\frac{1}{2n}\, \sum_{k=0}^{2n -1}
\frac{P(\zeta_k)}{(\zeta_k - z_0)^j},  \ \textrm{where} \ \zeta_k= z_0 +
z\omega^k.\] 
\end{proposition}
\begin{proof} Writing $P(\zeta)=P(z_0)+ P'(z_0)(\zeta-z_0) + \cdots +
\frac{P^{(n)}(z_0)}{n!}(\zeta-z_0)^n$ and introducing 
\[\zeta_k = z_0 + z\omega^k,\ \textrm{with}\ z \neq 0 \ \textrm{and}\
0\leq k\leq 2n-1,\] 
we find
\[P(\zeta_k) = P(z_0) + P'(z_0)z\omega^k + \cdots +
\frac{P^{(n)}(z_0)}{n!}z^n\omega^{kn}\] 
and then,
\[\frac{P(\zeta_k)}{(\zeta_k -z_0)^j}\, = \,\frac{P(\zeta_k)}{z^j\omega^{kj}} = P(z_0)z^{-j}\omega^{-kj} + \cdots +\frac{P^{(j)}(z_0)}{j\,!}
+ \cdots +\frac{P^{(n)}(z_0)}{n!}z^{n-j}\omega^{k(n-j)}.\] 
It is trivial to verify the identity $\sum_{k=0}^{2n-1}\omega^{kl}= \frac{1 -\omega^{2nl}}{1-\omega^l}=0$, if $-n \leq l\leq n$ and $l\neq 0$. Moreover, it is obvious that $\sum_{k=0}^{2n-1}\omega^{kl}=2n$, if $l=0$. Apply these with $l$ running over $\{-j, 1 -j,\ldots ,n-j\}$ to get finally 
\[\frac{P^{(j)}(z_0)}{j\,!} =  \frac{1}{2n}
\sum_{k=0}^{2n-1}\frac{P(\zeta_k)}{(\zeta_k-z_0)^j}. \]
\end{proof}

\

\section{The Open Mapping Theorem}

\begin{definition}\label{DEF71} Given $f:\Omega \to \mathbb C$, we say that $f$ is an \textit{open map} if for every open subset $O$ of $\Omega$, the set $f(O)=\{f(z): z \in O\}$ is an open set in the complex plane. 
\end{definition}

\begin{definition}\label{DEF72} Given $A$ and $B$, two subsets of $\mathbb C$, the \textit{distance} between $A$ and $B$ is the non-negative number $d(A;B)=\inf\{|a-b|:\, a \in A \
\textrm{and}\ b \in B\}$. 
\end{definition}

For another power series proof of the Open Mapping Theorem, we refer the reader to Cater ~\cite{CAT}. See also Lang ~\cite[p. 81]{LA}.

\begin{theorem}\label{TEO73} ({\sf Open Mapping Theorem}) Let $f$ be a non-constant analytic function in an open connected set containing $\overline{D}(0;R)$. Then, the image of $f$ contains an open disk centered at $f(0)$. 
\end{theorem}
\begin{proof} By the principle of isolated zeros (Proposition \ref{PROP23}), there exists a circle $C(0;r)=\{z:|z|=r\}$, with $r>0$, contained in $D(0;R)$ such that $f(0)\notin
f\big(C(0;r)\big)$. Then, we consider the strictly positive distance 
\[\delta = d\Big(f(0);\,f\big(C(0;r)\big)\Big).\]
Let us verify the inclusion $D\big(f(0);\delta/2\big) \subset f\big(D(0;r)\big)$. Given an arbitrary point $T\in D\big(f(0);\delta/2\big)$, there exists a point $v\in \overline{D}(0;r)$
satisfying 
\[|T-f(v)| =  d\Big(T;f\big(\overline{D}(0;r)\big)\Big) .\] 
It is clear that $|T-f(v)|\leq |T-f(0)|<\delta/2$. From the triangle inequality it follows that $|f(v)-f(0)|\leq |f(v)-T| + |T-f(0)|< \delta$. Hence, $f(v)\notin
f\big(C(0;r)\big)$ and we obtain $|v|<r$. Let us consider the radius $\rho=r-|v|>0$.

From above we conclude that there exists $\varphi$ analytic in $D(0;\rho)$, with $\varphi(0)\neq 0$, and $k\in \{1,2,3,\ldots\}$ such that for all $z \in D(0;\rho)$ we have the system
\begin{displaymath}
\left\{\begin{array}{ll}
f(v +z) =f(v) + z^k\varphi(z),\\
|T-f(v+z)|^2\geq |T-f(v)|^2 .
\end{array}
\right.
\end{displaymath}
Substituting the equation in the first line into the inequality in the second line and then expanding the resulting inequality we arrive at 
\[|T-f(v)|^2 -
2\textrm{Re}\Big\{\big[T-f(v)\big]\overline{z^k\varphi(z)}\Big\} \,+\,
|z|^{2k}|\varphi(z)|^2  \geq  |T-f(v)|^2, \ \textrm{if}\ z \in D(0;\rho),  \] 
which implies that, after cancelling the term $|T-f(v)|^2$, then substituting $z=re^{i\theta}$, where $0<r<\rho$ and $\theta \in \mathbb R$, and cancelling $r^k$, 
\[ -
2\textrm{Re}\Big\{\big[T-f(v)\big]\overline{\varphi(re^{i\theta})}e^{-ik\theta}\Big\} 
+ r^k|\varphi(re^{i\theta})|^2 \ \geq \ 0, \ \textrm{if} \ 0<r<\rho \
\textrm{and}\ \theta \in \mathbb R.\] 
Now, fixing the angle $\theta$ and letting $r\to 0^{+}$ we obtain
\[ - 2\textrm{Re}\Big\{\big[T-f(v)\big]\overline{\varphi(0)}e^{-ik\theta}\Big\}  \geq  0.\]
Since $\theta$ is arbitrary, and $k\geq 1$, we deduce 
that $[T-f(v)]\overline{\varphi(0)}=0$. Therefore, $T=f(v)$.
\end{proof}

\begin{remark}\label{REM74} One can easily show that the open mapping theorem implies the maximum and the minimum modulus principles. The famous Carath\'{e}odory's proof of the open mapping theorem for holomorphic functions can be seen in Carath\'{e}odory ~\cite[pp.~139--140]{CAR}, Burckel ~\cite[p.~172]{BU}, and Remmert ~\cite[pp.~256--258]{RR1} 
(see also Bak and Newman ~\cite[pp.~93--94]{BN}). 
\end{remark}

\section{The Inverse Function Theorem and The Local Representation Theorem}
\begin{theorem}\label{TEO81} Let us consider $f\in \mathcal{A}(\Omega)$ and $z_0\in \Omega$ satisfying
$f'(z_0)\neq 0$. Then, there exists $r>0$ such that  
\begin{itemize}
\item[(a)] $f$ is injective in $D(z_0;r)\subset \Omega$. 
\item[(b)] $f\big(D(z_0;r)\big)=V$ is an open set containing $f(z_0)$.
\item[(c)] $\varphi=f|_{D(z_0;r)}:D(z_0;r) \to V$ is invertible and its
inverse is continuous.
\item[(d)] $\varphi^{-1}$ is complex-differentiable.
\end{itemize}
\end{theorem}
\begin{proof}  Replacing $f$ with the map $z\mapsto f(z+z_0)-f(z_0)$, allows us to
suppose that $f(z_0)=0$ and $z_0=0$.  
\begin{itemize}
\item[(a)] Let $r>0$ be such that $|f'(\zeta)|>|f'(0)|\big/ 2$ for all
$\zeta\in D(0;2r)\subset \Omega$ and such that the Taylor series (with
coefficients $a_n$) of $f$ around the origin converges in $D(0;2r)$. Given $z$
and $w$, both in $D(0;r)$, we find 
\[f(z)-f(w)= a_1(z-w)+a_2(z^2-w^2)+ \cdots + a_n(z^n -w^n) + \cdots \]
and, also supposing $z\neq w$, 
\[\frac{f(z)-f(w)}{z-w}= f'(0) +
\sum_{n=2}^{+\infty}a_n\sum_{k=0}^{n-1}z^{n-1-k}w^k . \] 
We have $\big|\sum_{n=2}^{+\infty}a_n\sum_{k=0}^{n-1}z^{n-1-k}w^k\big|\leq
\sum_{n=2}^{+\infty}|a_n|nr^{n-1}$, and 
the series $f'(z)=\sum_{n\geq 1} na_nz^{n-1}$ is absolutely convergent in the open disk $D(0;2r)$. Clearly, we can suppose $r$ small enough (with $r$ strictly positive) so that $\sum_{n\geq 2} n|a_n|r^{n-1}<|\,f'(0)\,|\big/2$. Thus, given two arbitrary points $z$ and $w$, both within $D(0;r)$, it follows that 
\[\Big|\frac{f(z)-f(w)}{z-w}\Big|>\frac{|f'(0)|}{2}>0 \ \ \textrm{and thus }\
f(z)\neq f(w).\] 
\item[(b)] It follows from the open mapping theorem (Theorem \ref{TEO73}), by noticing that $D(0;r)$ is
connected. 
\item[(c)] Since $\varphi$ is bijective and open, if $O$ is open then
the set $\big(\varphi^{-1}\big)^{-1}(O)=\varphi(O)$ is also open. As a consequence, $\varphi^{-1}$ is continuous.  
\item[(d)] In the Newton quotient $\frac{\varphi^{-1}(w)-\varphi^{-1}(w_0)}{w-w_0}$ 
we make the notational change $\varphi^{-1}(w)=z$, $\varphi^{-1}(w_0)=z_0$,
$w=f(z)$, and $w_0=f(z_0)$. Since the function $\varphi^{-1}$ is continuous, it follows that
$z \to z_0$ if $w \to w_0$ and consequently 
\[\big(\varphi^{-1}\big)'(w_0)=\lim\limits_{w \to
w_0}\frac{\varphi^{-1}(w)-\varphi^{-1}(w_0)}{w-w_0}=\lim_{z\to z_0}\frac{z -
z_0}{f(z)-f(z_0)}=\frac{1}{f'(z_0)}.\] 
\end{itemize}
\end{proof}

\begin{theorem}\label{TEO82} ({\sf Local Representation}) Let $f$ be analytic in $D(z_0;r)$, with $r>0$ and  
\[f(z)= a_0+\sum_{n=m}^{+\infty}a_n(z-z_0)^n, \ \textrm{where}\ m\ \textrm{is fixed},\ m\geq 1,\ \textrm{and}\ a_m\neq 0.\]
Then, there exists $\varphi$ analytic and bijective in an open disk centered at the origin such that $\varphi$ has differentiable inverse and satisfies 
\[f(z)= a_0+\varphi(z-z_0)^m .\]
\end{theorem}
\begin{proof} We write $f(z)-a_0=a_m(z-z_0)^m +a_{m+1}(z-z_0)^{m+1}+ \cdots $ as
\[f(z)-a_0= a_m(z-z_0)^m[1 + g(z-z_0)], \ \textrm{with}\ g(0)=0. \]
Let $a\in \mathbb C$ be a mth root of $a_m$ (i.e., $a^m=a_m$). Then, by Proposition \ref{PROP21}, on complex binomial series, we can consider a convergent power series
$G(z)$ satisfying, for $|z|$ sufficiently small, 
\[1 + g(z) = [1+G(z)]^m, \ \textrm{with}\ G(0)=0.\]
Then we have  
\[f(z)-a_0=\big\{a(z-z_0)[1+G(z-z_0)]\big\}^m,\]  
implying that
\[f(z) = a_0 + \varphi(z-z_0)^m, \ \textrm{where}\ 
\varphi(w)=aw\big(1+G(w)\big).\] 
Since $G(0)= 0$, the second coefficient of the Taylor series of $\varphi$ centered at $z=0$ is $a\neq 0$. Thus, by the inverse function theorem (Theorem ~\ref{TEO81}),
$\varphi$ enjoys the desired properties. 
\end{proof}

\

\section{The Theorem of Read and Connell-Porcelli and the Theorem of Hurwitz}

This section presents an adaptation for analytic functions of the proof of a theorem for holomorphic functions given by Whyburn in ~\cite{WH1} and ~\cite[pp.~81--82]{WH2}. This theorem was independently proved, in 1961, by A. H. Read ~\cite{REA} and by E. H. Connell and P. Porcelli ~\cite{COPO1}.  A proof by A. Hurwitz of a version of this theorem for analytic functions can be found in ~\cite{HU}.

\begin{theorem}\label{TEO91} Let $f$ be analytic in $D(0;1)$. Then, there exists a sequence $(a_j)\subset \mathbb C$ such that
\[f(z) =\sum a_jz^j,\  \textrm{for all} \ z \in D(0;1).\]
\end{theorem}
\begin{proof} Let us fix an arbitrary $n\geq 1$, with $n\in \mathbb N$. Let $g$ be the analytic function in $D(0;1)$ given by
\[g(z) = \sum_{k=0}^{2n-1}(-1)^kf(z\omega^k)\,,\ \textrm{where}\  \omega=e^{i\pi/n}\ \ \textrm{(thus, $\omega^n=-1$)}.\]
From the hypothesis it follows that there exists a complex sequence $(a_j)$ satisfying $f(z)=\sum a_j z^j$, for all $z \in D(0;\rho)$, for some $0<\rho<1$. Hence, there exists a complex sequence $(b_j)$ such that we have $g(z)= \sum b_j z^j$, for all $z \in D(0;\rho)$. In addition, 
\begin{displaymath}
\left\{\begin{array}{ll}
b_j = \sum\limits_{k=0}^{2n-1}(-1)^ka_j\omega^{jk} = a_j \frac{1 -\omega^{j2n}}{1 +\omega^j} = 0,  \ \textrm{if} \ 0\leq j<n,\\
\\
b_n  =  \sum\limits_{k=0}^{2n-1}(-1)^ka_n\omega^{nk} =2na_n .
\end{array}
\right.
\end{displaymath}
Therefore, we can write $\sum_{j=0}^{+\infty} b_jz^j = z^n\Big(2na_n + \sum_{j=n+1}^{+\infty} b_jz^{j-n}\Big)$, where $z$ belongs to $D(0;\rho)$.
Now, it is easy to see that the function 
\begin{displaymath}
\varphi(z)\ \ =\ \ \left\{\begin{array}{ll}
2na_n + \sum\limits_{j=n+1}^{+\infty} b_jz^{j-n}, & \textrm{if}\ z \in D(0;\rho),\\
\frac{g(z)}{z^n}, & \textrm{if}\ z \neq 0,  
\end{array}
\right.
\end{displaymath}
is well defined and analytic in $D(0;1)$. Putting $M(r)=\max\{|f(z)|: |z|=r\}$, with $0<r<1$, by Theorem \ref{TEO41} (the maximum modulus principle) we infer that
 \[2na_n=|\varphi(0)|\leq \max\limits_{|z|=r}|\varphi(z)|=\max\limits_{|z|=r}\frac{|g(z)|}{r^n}\leq \frac{2nM(r)}{r^n} .\]
Therefore, we have proved that
\[|a_n|r^n \leq M(r)\,, \ \ \ \textrm{if}\ 0<r<1\,,\ \textrm{for all}\ n \in \mathbb N.\]
Hence, if $R$ is the radius of convergence of $\sum a_nz^n$, from the Cauchy-Hadamard formula follows that
\[R^{-1}=\limsup\sqrt[n]{|a_n|}\leq \limsup\frac{\sqrt[n]{M(r)}}{r}= r^{-1},\]
which implies $R\geq r$ for all $0<r<1$ and then, $R \geq 1$. As a consequence, the power series $\sum a_nz^n$ converges within $D(0;1)$ and, through the principle of isolated zeros (Proposition \ref{PROP23}) we conclude that $f(z)=\sum a_nz^n$, for all $z \in D(0;1)$.
\end{proof}

\begin{remark} From Theorem ~\ref{TEO91} and Liouville's theorem (Theorem ~\ref{TEO312}) we obtain a very easy proof of the fundamental theorem of algebra. In fact, let us suppose that there exists a polynomial $p=p(z)$, degree $(p)\geq 1$, with no zeros. Hence, $1/p(z)$ is entire and by Theorem ~\ref{TEO91} together with point 5 on p. 3, we have $1/p(z) = \sum a_nz^n$, for all $z \in \mathbb C$. Moreover, since $|p(z)|\to +\infty$ as $|z|\to +\infty$ (see \cite{OO1}), there exists $M \in \mathbb R$ such that $|\sum a_nz^n|=1/|p(z)|\leq M$, for  all $z \in \mathbb C$. Finally, from Theorem \ref{TEO312} we conclude that $\sum a_nz^n$ is a constant and degree$(p)=0$, which is a contradiction.
\end{remark}

\

\section{The Schwarz Lemma and the Clunie-Jack Lemma}

Section 10 provides a modest generalization of the Schwarz Lemma for analytic functions, as a consequence of the Gutzmer-Parseval inequality for analytic functions (Theorem ~\ref{TEO33}). This generalization is inspired by a power series proof of the Schwarz Lemma for complex-differentiable functions, given by Burckel ~\cite[p.~191]{BU}. Erhardt Schmidt's famous proof of Schwarz's Lemma for holomorphic functions, first published by Carath\'{e}odory, in 1905 (see Osserman ~\cite{OSS}), can be seen in Carath\'{e}odory ~\cite[pp. ~135--136]{CAR}. See also Bak and Newman ~\cite[p.~94]{BN}, Conway ~\cite[pp.~130--131]{CON}, Lang ~\cite[pp.~210--211]{LA}, and Remmert ~\cite[p.~270]{RR1}.  

As a result of Schwarz's Lemma, this section presents a trivial proof of the Clunie-Jack Lemma (originally published in 1971), see Jack ~\cite{JA}. See also H. P. Boas ~\cite{BO}, Burckel ~\cite[p. 207]{BU}, Osserman ~\cite{OSS}, and P\'{o}lya and  Szeg\"{o} ~\cite[Part III, Problem ~291]{POSZ}.

Here and subsequently, $S^1$ stands for the set $\{z\in \mathbb C: |z|=1\}$.

\begin{theorem}\label{TEO101} ({\sf Schwarz's Lemma}) Let $f$ be analytic in $D(0;1)$ such that
\[  |f(z)|\leq 1,  \ \textrm{for all}\ z \in D(0;1),  \ \textrm{and}  \  f(0)=0.  \]
The following statements are true.
\begin{itemize}
\item[$\bullet$] We can write $f(z)=\sum_{n=1}^{+\infty}a_nz^n$, where
\[  \    \sum\limits_{n=1}^{+\infty}|a_n|^2\leq 1 \ \    \textrm{and}    \  \ |f(z)|\leq |z|,\  \textrm{for all}\ z\ \in D(0;1).\ \ \  \]
\item[$\bullet$] We have $|a_n|=1$, for some $n\geq 1$, if and only if there is $\omega\in S^1$ satisfying
\[f(z)=\omega z^n\,, \  \textrm{for all}\  z \in D(0;1).  \]
\item[$\bullet$] We have $|f(z)|=|z|$, for some $z\neq 0$, if and only if there is $\omega\in S^1$ satisfying 
\[f(z)=\omega z\,, \  \textrm{for all}\  z \in D(0;1).  \]
\end{itemize}
\end{theorem}
\begin{proof} Since $f(0)=0$, through using Theorem ~\ref{TEO91} (by Hurwitz, Connell-Porcelli, and Read) we write $f(z)=\sum_{n=1}^{+\infty}a_nz^n$, where $z \in D(0;1)$ and $a_n \in \mathbb C$, for all $n\geq 1$. As a consequence, since $|f(z)|\leq 1$ for all $z\in D(0;1)$, from the Gutzmer-Parseval inequality for analytic functions (Theorem \ref{TEO33}) we obtain the inequality 
\[ \sum_{n=1}^{+\infty}|a_n|^2r^{2n}\leq 1\,, \ \textrm{for all}\ r \ \textrm{in}\ [0,1)\,,\]
and taking the limit of such an inequality  as $r \to 1^{-}$ we find that $\sum_{n=1}^{+\infty}|a_n|^2\leq 1$. 

Thus, if $|a_n|=1$ for a particular $n\geq 1$, then we have $f(z)=a_nz^n$, for all $|z|<1$.

Now, let us consider any point $\zeta$ in the circle $\{z:|z|=r\}$, with $r$ fixed and $0<r<1$. It is clear that $|\sum_{n=1}^{+\infty} a_n\zeta^{n-1}|= r^{-1}|\sum_{n=1}^{+\infty} a_n\zeta^n|\leq r^{-1}$. Therefore, from the maximum modulus principle for analytic functions (Theorem \ref{TEO41}) it follows the inequality $|\sum_{n=1}^{+\infty} a_nz^{n-1}|\leq 1/r$, for all $z \in \overline{D}(0;r)$. Consequently, as $r$ is arbitrary in the interval $(0,1)$, we find that $|\sum_{n=1}^{+\infty} a_nz^{n-1}|\leq 1$, for all  $z \in D(0;1)$. 

Hence, we deduce that $|f(z)|=|z\sum_{n=1}^{+\infty} a_nz^{n-1}|\leq |z|$, for all $z \in D(0;1)$.

If $|\sum_{n=1}^{+\infty} a_n\zeta^n|=|\zeta|$, for some $\zeta$ such that $0<|\zeta|<1$, then we have the identity $|\sum_{n=1}^{+\infty} a_n\zeta^{n-1}|=1$. We already proved (two paragraphs above) the inequality $|\sum_{n=1}^{+\infty} a_nz^{n-1}|\leq 1$, for all $z$ inside $D(0;1)$. Therefore, employing the maximum modulus principle for analytic functions (Theorem ~\ref{TEO41}) we deduce the identity $\sum_{n=1}^{+\infty} a_nz^{n-1}=a_1\in S^1$, for all $z \in D(0;1)$. Thus, we conclude that $f(z)= \sum_{n=1}^{+\infty} a_nz^n=a_1z$, for all $z \in D(0;1)$.
\end{proof}

\newpage

\begin{theorem}\label{TEO102}({\sf Clunie-Jack Lemma})
Let $f:\overline{D}(0;1)\to \overline{D}(0;1)$ be non-constant and
analytic (see Definition \ref{DEF51}). The following statements hold.
\begin{itemize}
\item[$\bullet$] Supposing that $|f|$ has a maximum at $\alpha \in S^1$, then we have
\[\alpha\frac{f'(\alpha)}{f(\alpha)} >0 .\]
\item[$\bullet$] If all the conditions above are true and we also have $f(0)=0$, then
\[\alpha\frac{f'(\alpha)}{f(\alpha)} \geq 1 .\]
\end{itemize}
\end{theorem}
\begin{proof} Since $f$ is non-constant, we have $f(\alpha)\neq
0$. From the anti-calculus proposition (Theorem \ref{TEO52}) it follows
that $f'(\alpha)\neq 0$. Considering the function  
\[g(z)= \frac{f(\alpha z)}{f(\alpha)}\,,\ \textrm{where}\ z \in \overline{D}(0;1),\] 
we have $g(1)=1$ and $|g|$ attains its maximum value $1$ at the point $z=1$. Moreover,
in a small open disk centered at $z=1$ we have 
\begin{equation}\label{EQ10.1}
g(z) = 1 + \varphi(z)(z-1), \ 
\textrm{where}\ \varphi (z) \to g'(1) \ \textrm{if}\ z \to 1, \
\textrm{with}\ z \in D(0;1).
\end{equation}
We also have
\begin{equation}\label{EQ10.2}
|g(z)|^2 \leq 1=|g(1)|^2, \  \textrm{for all} \ z \in \overline{D}(0;1). 
\end{equation} 
Substituting into (\ref{EQ10.1}) the expression $z = 1 + re^{i\theta}$, with $\theta
\in  (\frac{\pi}{2},\frac{3\pi}{2})$ and $r=r_\theta$ strictly positive and small enough so that
$z \in D(0;1)$, and the expression so obtained into (\ref{EQ10.2}), we find 
\[1  +  2\textrm{Re}\big[\varphi(1 + re^{i\theta})re^{-i\theta}\big]\, +\,
r^2|\varphi(1 + re^{i\theta})|^2  \leq  1.\] 
Cancelling $1$ on each side of the inequality right above and then dividing by $r>0$, we
arrive at 
\[ 2\textrm{Re}\big[\varphi(1 + re^{i\theta})e^{-i\theta}\big]\, +\,
r|\varphi(1 + re^{i\theta})|^2 \leq 0  .\] 
Now, fixing $\theta \in (\frac{\pi}{2},\frac{3\pi}{2})$ and letting $r \to
0^{+}$ we find, since $\varphi(1) =g'(1)$, 
\begin{equation}\label{EQ10.3}
2\textrm{Re}\big[g'(1)e^{-i\theta}\big] \leq 0  .
\end{equation}
Since $\theta$ is arbitrary in $(\frac{\pi}{2}\,,\frac{3\pi}{2})$ and the
expression in (\ref{EQ10.3}) is continuous in $\theta$, we obtain
\begin{displaymath}
\left\{\begin{array}{ll}
-2\textrm{Re}\big[g'(1)\big]\ \leq 0,\ &  \ \textrm{if}\ \theta =\pi,\\
-2\textrm{Im}\big[g'(1)\big]\ \leq 0,\ & \ \textrm{if}\  \theta \to \frac{\pi}{2}^{+},\\
+2\textrm{Im}\big[g'(1)\big]\ \leq 0,\ & \ \textrm{if}\ \theta \to \frac{3\pi}{2}^{-} .\\
\end{array}
\right.
\end{displaymath}
Therefore, the number $g'(1)= \alpha \frac{f'(\alpha)}{f(\alpha)}$ is real and strictly
 positive. This proves the first statement.

To prove the second statement we notice that since $f(0)=0$, we have $g(0)=0$. As a consequence, employing
the Schwarz Lemma (Theorem \ref{TEO101}) we conclude that $|g(z)|\leq |z|$, for all $z \in \overline{D}(0;1)$.
Finally, given an arbitrary $t$ in $(0,1)$ we have $|g(t)-g(1)|\geq 1 -|g(t)|\geq 1-t$ and then
\[|g'(1)| =\left|\lim_{t \to 1^{-}}\frac{g(t)-g(1)}{t-1}\right|\geq \lim_{t\to 1}
\left|\frac{1-t}{t-1}\right|=1.\]
\end{proof}

\section{The Weierstrass Double Series Theorem} 

\begin{theorem}\label{TEO111}({\sf Weierstrass's Double Series Theorem}) Let
$\sum_{\mu=0}^{+\infty}f_\mu(z)$ be a series of convergent power series
$f_\mu(z) = \sum _{n= 0}^{+\infty}a_n(\mu)z^n$, where $\mu \in \mathbb N$, in the disk $D(0;R)$, where $R>0$, with coefficients $a_n(\mu) \in \mathbb C$. Let us suppose that the series $\sum f_\mu$ converges uniformly in
$\overline{D}(0;\rho)$, for each $\rho$ such that $0<\rho<R$, to the function $F(z)=\sum_{\mu=0}^{+\infty}f_\mu(z)$ in $D(0;R)$. Then, for all $z \in D(0;R)$ and all $k \in \mathbb N$ we have 
\[F(z) = \sum_{n=0}^{+\infty} \sum_{\mu=0}^{+\infty}a_n(\mu)z^n\ \ \  
\textrm{and} \ \  F^{(k)}(z)= \sum_{\mu=0}^{+\infty}f_\mu^{(k)}(z),\]
with uniform convergence in every closed disk $\overline{D}(0;\rho)$, where
 $0<\rho<R$. 
\end{theorem}
\begin{proof} Let us fix $r$, where $\rho<r<R$. By hypothesis, given
$\epsilon>0$ there exists $\mu_0=\mu_0(\epsilon)\geq 0$ such that
\[|f_{\mu+1}(z)+ \cdots +f_{\mu+p}(z)|\leq \epsilon, \  \textrm{for all}\ \mu\geq
\mu_0,\ \textrm{all}\  |z|\leq r,\ \textrm{and all}\ p\in \mathbb N .\] 
Then, from Theorem \ref{TEO33} (the Gutzmer-Parseval inequality for analytic functions)
follows 
\[\sum_{n=0}^{+\infty}|a_n(\mu+1) + \cdots +a_n(\mu +p)|^2\,|z|^{2n}\leq
 \epsilon^2,\ \textrm{for all}\ \mu\geq \mu_0,\ \textrm{all} \ |z|\leq r,\ \textrm{and all} \ p\in \mathbb N, \] 
and in this way we have [for $z \in \overline{D}(0;\rho)$]
\begin{displaymath} 
(11.1) \ \ \ \ \ \ \ \ \ \left\{\begin{array}{ll}
 |a_n(\mu+1) + \cdots +a_n(\mu +p)|\,|z|^n = \\
= \sqrt{|a_n(\mu+1)
+ \cdots +a_n(\mu +p)|^2\,|r|^{2n}}\, \frac{|z|^{\,n}}{r^n}\leq
\epsilon\big(\frac{\rho}{r}\big)^n,\\ 
 \textrm{for all} \ \mu\geq \mu_0=\mu_0(\epsilon), \ \textrm{all} \ |z|\leq \rho, \ \textrm{all}\ p\in \mathbb N,\ \textrm{and all} \ n\in \mathbb N. 
\end{array}
\right. \ \   \ \ \ \ \ \ \ \ \ \ \ \ \ \ \ \ \ \
\end{displaymath}

Hence, for any $n\in \mathbb N$ and any $z\in \overline{D}(0;\rho)$, we proved
that given $\epsilon >0$ then there exists $\mu_0=\mu_0(\epsilon)$ such that  
$|a_n(\mu+1)z^n + \cdots +a_n(\mu+p)z^n|\leq \epsilon$ (since
$\rho/r\leq 1$), for all $\mu\geq \mu_0$ and all $p\in \mathbb N$. Thus, the series $\sum_{\mu=0}^{+\infty} a_n(\mu)z^n$ fulfills the well-known Cauchy's Criterion for numerical
series and converges and, therefore, so does $\sum_{\mu=0}^{+\infty}
a_n(\mu)$. Next, we return to (11.1). 

Letting $p\to +\infty$ at (11.1), and using the index $\nu \in \mathbb N$ to
label a sequence of partial sums, we find the inequalities 
$\big|\,\sum_{\mu =0}^{+\infty}a_n(\mu) \,-\, \sum_{\mu =0}^\nu
a_n(\mu)\,\big|\,|z|^n \leq\epsilon (\rho/r)^n$, for all $\nu\geq \mu_0$,
all $|z|\leq \rho$, and all $n\in \mathbb N$. Summing up these
inequalities over $n\in \mathbb N$, we obtain 
\[\sum_{n=0}^{+\infty}\Big|\sum_{\mu =0}^{+\infty}a_n(\mu) - \sum_{\mu
=0}^\nu a_n(\mu)\Big|\,|z|^n \leq \frac{\epsilon}{1 - \frac{\rho}{r}}, \ 
\textrm{for all} \ \nu\geq \mu_0 \ \textrm{and all} \ |z|\leq \rho .\]  
As a consequence, given an index $\nu\geq \mu_0$ and a point $z \in \overline{D}(0;\rho)$, we have the inequality
$\big|\sum_{n=0}^{+\infty}\big[\sum_{\mu =0}^{+\infty}a_n(\mu) - \sum_{\mu
=0}^\nu a_n(\mu)\big]z^n\big|\leq \epsilon r/(r -\rho)$, which entails   
\[\Big|\sum_{n=0}^{+\infty}\sum_{\mu =0}^{+\infty}a_n(\mu)\,z^n  - \sum_{\mu
=0}^\nu f_\mu(z)\Big| \leq \frac{\epsilon r}{r -\rho},\  \textrm{for all}\  \nu\geq
\mu_0(\epsilon) \ \textrm{and all}\ |z|\leq \rho.\] 
Taking $\rho$ arbitrarily close to $R$ shows that
$F(z)=\sum_{\mu=0}^{+\infty}f_\mu(z) =  \sum_{n=0}^{+\infty}\sum_{\mu
=0}^{+\infty}a_n(\mu)\,z^n$, for every $z \in D(0;R)$, with the convergence
uniform over $\overline{D}(0;\rho)$, if $0<\rho<R$. 

Finally, putting $s_\nu= f_0+ \cdots +f_\nu$, where $\nu\in \mathbb N$, we see that the sequence $(s_\nu - F)_{\nu \in \mathbb N}$ converges uniformly to the zero function over $\overline{D}(0;\rho)$, if $0<\rho<R$, and then by Corollary \ref{COR37} the sequence $(s_\nu' -F')_{\nu \in \mathbb N}$ also does. Proceeding by induction on $k$ we obtain the identities $\sum_{\mu \geq 0}^{+\infty} f^{(k)}_\mu (z)= F^{(k)}(z)$, for arbitrary $k \in \mathbb N$ and $z\in D(0;R)$, with uniform convergence over all the compact disks $\overline{D}(0;\rho)$, where $0<\rho<R$. 
\end{proof}

\begin{notation}\label{DEF112} Let us consider $X$ a nonempty subset of $\mathbb C$.
\begin{itemize}
\item[$\circ$] We denote by $C(X)$ the set $\{f:X\to \mathbb C\,, \ \textrm{where}\ f\ \textrm{is continuous}\}$. 
\item[$\circ$] Given $K$ a nonempty compact subset of $X$ and $f\in C(X)$, we put 
\[|f|_K= \sup_{z \in K}|f(z)| = \max_{z \in K}|f(z)| .\]
\end{itemize}
\end{notation}

The number $|f|_K$ is called the \textit{ norm} (the sup norm) of $f$ over $K$.

Given a sequence $(f_n)$ in $C(X)$, a compact set $K\subset X$, and $f\in C(X)$, it is clear that $(f_n)$ converges uniformly to $f$ on $K$ if and only if $|f_n -f|_K \to 0$ as $n\to +\infty$.

\begin{definition}\label{DEF113} A sequence, or a series, of functions in $C(X)$ \textit{converges compactly} on $X$ if it converges uniformly on every compact subset of $X$. 
\end{definition}

\begin{definition}\label{DEF114} Given $X\subset \mathbb C$, we say that \begin{itemize}
\item[$\circ$] $L$ is a \textit{ compact neighborhood} of $X$ if $L$ is compact and there exists an open set $V$ such that $X\subset V \subset L$. 
\item[$\circ$] the \textit{ boundary} of $X$ is  
\[\partial X  =\big\{\zeta \in \mathbb C: D(\zeta ;r) \cap X\neq \emptyset
\ \ \textrm{and}\ \ D(\zeta;r)\cap\big(\mathbb C\setminus X\big)\neq
\emptyset, \ \textrm{for all}\ r>0 \big\} . \] 
\end{itemize}
\end{definition}

\begin{corollary}\label{COR115} Let $(f_n)_{n \in \mathbb N}$ be a sequence in $\mathcal{A}(\Omega)\cap C(\overline{\Omega})$, with $\Omega$ connected and bounded, such that the sequence $(f_n|_{_{\partial \Omega}})_{n\in \mathbb N}$ converges uniformly on $\partial \Omega$. Then, 
\begin{itemize}
\item[(a)] For every $k\in \mathbb N$, the sequence $(f_n^{(k)})_{n\in\mathbb N}$ converges compactly on $\Omega$. 
\item[(b)] If $f=\lim f_n$, then $f\in \mathcal{A}(\Omega)\cap
C(\overline{\Omega})$. 
\item[(c)] The sequence $(f_n^{(k)})_{n\in \mathbb N}$ converges compactly to $f^{(k)}$ on $\Omega$, for every $k\in \mathbb N$. 
\end{itemize}
\end{corollary}
\begin{proof} By employing the maximum modulus principle (Theorem \ref{TEO41}) we deduce the identities $|f_n-f_m|_{\overline{\Omega}}= |f_n -f_m|_{\partial \Omega}$, for all $n\in \mathbb N$ and all $m\in \mathbb N$. Therefore, the sequence $(f_n)$ converges uniformly on $\overline{\Omega}$ to a function $f\in C(\overline{\Omega})$. 

Next, let us consider an arbitrary compact disk $\overline{D}(z_0;r)\subset \Omega$, where $r>0$. From Theorem ~\ref{TEO91} (by Connell-Porcelli, Hurwitz, and Read) we know that throughout the disk $D(z_0;r)$ each function $f_n$ is given by its Taylor series centered at $z_0$. In addition, the nth partial sum of the series 
\[f_1 + \sum_{n=1}^{+\infty}(f_{n+1}-f_n)\]
is $s_n= f_1+(f_2-f_1)+ \cdots +(f_{n+1}-f_n)= f_{n+1}$ and, by the previous paragraph, the sequence $(s_n)$ and the series $f_1 + \sum_{n=1}^{+\infty}(f_{n+1}-f_n)$ both converge uniformly to $f$ on $\overline{D}(z_0;r)$. From the Weierstrass double series theorem (Theorem \ref{TEO111}) it follows that $f$ is analytic in $D(z_0;r)$ and $f_n^{(k)}\to f^{(k)}$ compactly on
$D(z_0;r)$, for every $ k\in \mathbb N$. Finally, through a simple compactness argument we infer that the sequence $(f_n^{(k)})_{n\in \mathbb N}$ converges compactly to $f^{(k)}$ on $\Omega$, for every $k \in \mathbb N$. 
\end{proof}

\section{Montel's Theorem}

For the sake of completeness, in this section we present a proof of Montel's Theorem for analytic functions. This demonstration employs Corollary ~\ref{COR36} (Cauchy's Inequalities) and Corollary ~\ref{COR37}, both following the Gutzmer-Parseval inequality for analytic functions (Theorem ~\ref{TEO33}). For additional power series proofs of Montel's Theorem, see Narasimhan and Nievergelt ~\cite[pp.~34--35]{NN} and Read ~\cite{REA}.

\begin{definition}\label{DEF121} A family $\mathcal{F}$ contained in $C(\Omega)$ is 
\begin{itemize}
\item[$\circ$] \textit{ normal} if every sequence in $\mathcal{F}$ contains a subsequence compactly convergent to a function $f$ [it is clear that $f\in C(\Omega)$; it is not required that $f\in\mathcal{F}$].   
\item[$\circ$] \textit{ locally bounded} if for every $z_0\in \Omega$ there
 exists an open disk $D(z_0;r)$ and a finite constant $M$ such that
$|f(z)|\leq M$, for all $f \in \mathcal{F}$ and all $z
\in D(z_0;r)$.
\item[$\circ$] \textit{ equicontinuous} on $X\subset\Omega$ if for every
$\epsilon>0$ there exists $\delta>0$ such that
\[|f(z)-f(w)|<\epsilon,\ \forall f \in \mathcal{F}\ \textrm{and}\ \forall
z \ \textrm{and}\ \forall w, \ \textrm{both in}\  X, \ \textrm{such
that}\ |z-w|<\delta .\]  
\end{itemize}
\end{definition}
It is easy to verify that if $\mathcal{F}$, where $\mathcal{F}\subset C(\Omega)$, is locally bounded and $K$ is compact in $\Omega$, then there exists $M\in \mathbb R$ such that
$|f(z)|\leq M$, for all $f\in \mathcal{F}$ and for all $z \in K$. We then
say that $\mathcal{F}$ is \textit{uniformly bounded on the compact subsets of $\Omega$}.

If $\mathcal{F}$ is locally equicontinuous then $\mathcal{F}$ is
equicontinuous on the compacta in $\Omega$. 

\begin{lemma}\label{LEM122} Let us consider the countable collection of open disks
\[\mathcal{C}=\big\{D(a_n;r_m): a_n\in \mathbb Q\times \mathbb 
Q\ \textrm{and}\ r_m\in \mathbb Q, \ \textrm{where}\ r_m>0, \    n \in \mathbb N,\ \textrm{and}\ m\in \mathbb N\big\}.\] 
Then, every open set in $\mathbb R^2$ is a union of sets in $\mathcal{C}$.
\end{lemma}
\begin{proof} Let $\Omega$ be an arbitrary open set in $\mathbb R^2$ and
$D(z;2r)$ an open disk contained in $\Omega$, with $r$ a strictly positive
rational number. It is clear that there exists a point $w\in D(z;r)\cap(\mathbb Q\times\mathbb Q)$. Moreover, it is easy to see that $z\in D(w;r)\subset D(z;2r)\subset\Omega$. We complete the proof by noticing that $D(w;r)\in \mathcal{C}$. 
\end{proof}

\begin{theorem}\label{TEO123} Let $\mathcal{F}$ be a locally bounded family in
$\mathcal{A}(\Omega)$. Then, \begin{itemize} 
\item[(a)] $\mathcal{F}$ is equicontinuous on each compact subset of $\Omega$.
\item[(b)]{\sf (Montel's Theorem)} $\mathcal{F}$ is normal. 
\end{itemize}
\end{theorem}
\begin{proof} Let us fix $K$, where $K$ is an arbitrary compact subset of $\Omega$. 
\begin{itemize}
\item[(a)] Let us pick
$r=d(K;\partial\Omega)/4>0$. Since $\mathcal{F}$ is locally bounded, $\mathcal{F}$ is uniformly bounded on the compact set $K(3r)=\{z : d(z;K)\leq3r\}\subset \Omega$. That is, there exists $M\in \mathbb R$ such that $|f(z)|\leq M$, for all $f\in \mathcal{F}$ and for all $z \in K(3r)$. 

 Given an arbitrary $z_0 \in K$, from Theorem ~\ref{TEO91} (by Hurwitz, Read, and Connell-Porcelli) we deduce that the Taylor series of an arbitrary
function $f\in \mathcal{F}$ around $z_0$, written as $f(z)=\sum_{n=0}^{+\infty}
a_n(z-z_0)^n$, converges in $D(z_0;4r)$. Therefore, given an arbitrary $h \in D(0;r)$ we have $f(z_0 +h) = \sum a_nh^n$ and
\begin{equation}\label{12.1}
|f(z_0+h)-f(z_0)| = \Big|\sum_{n\geq 1}a_nh^n\Big| \leq
|h|\,\sum|a_n|\,|h|^{n-1} \leq |h|\,\sum n|a_n||h|^{n-1} .
\end{equation}
Since $f'(z_0+h)=\sum na_nh^{n-1}$, by Corollaries ~\ref{COR36} and ~\ref{COR37} we have  
\begin{equation}\label{12.2}
n|a_n|\,(2r)^{n-1}\leq \max\limits_{\overline{D}(z_0;2r)}|f'|\leq \frac{M}{3r-2r}
=\frac{M}{r}, \  \textrm{for all}\ n \in \mathbb N .
\end{equation}
Using inequalities (\ref{12.1}) and (\ref{12.2}), after observing that $|h|/ 2r \leq  1/ 2$, we conclude the proof of (a) thus: 
\[ |f(z_0+h)-f(z_0)| \leq  |h|\sum
n|a_n|(2r)^{n-1}\left(\frac{|h|}{2r}\right)^{n-1}\leq |h|\frac{M}{r}\sum
\left(\frac{|h|}{2r}\right)^{n-1}\leq |h|\frac{2M}{r} ,\] 
valid for all $f\in \mathcal{F}$, all $z_0\in K$, and all $|h|<r$.
 
\item[(b)] Let us fix an arbitrary sequence $(f_n)_{n\in \mathbb N}$ in
$\mathcal{F}$. 

{\sf Claim 1.} There exists a subsequence of $(f_n)$ uniformly convergent on $K$.  

To prove this claim, let $X=\{x_k:k \in \mathbb N\}$ be dense in  $K$. Putting $ N_0=\mathbb N$, let us construct inductively a sequence of infinite sets of indexes $N_k\subset N_{k-1}$, where $k\in \{1,2,3,\ldots\}$. For fixed $k \geq 1$, the sequence $\big(f_n(x_k)\big)_{n\in N_{k-1}}$ is by hypothesis
bounded and therefore possesses a convergent subsequence, indexed by an
infinite set $N_k\subset N_{k-1}$. Then, if $n_p$ is the $p$th index in
$N_p$, the sequence $(f_{n_p}(x_k))_{p\in \mathbb N}$ converges, for each
$k\in \mathbb N$.  

Given $\epsilon>0$, we consider any $\delta>0$ following from the
equicontinuity of $\mathcal{F}$ in $K$. Then, for some $k \geq 1$ we have
$K\subset D(z_1;\delta)\cup ...\cup D(z_k;\delta)$. Let $N\in \mathbb N$ be
such that     
\[|f_{n_p}(z_j)-f_{n_q}(z_j)|<\epsilon\,,\ \ \textrm{if}\ j=1,...,k\
\textrm{and}\ p,q \geq N.\] 
Then, for fixed $w\in K$, choose $j$ so that $w\in D(z_j;\delta)$. Hence,
 for $p,q\geq N$ we have, by definition of $\delta$,
\[|f_{n_p}(w) - f_{n_q}(w)| \leq |f_{n_p}(w) -f_{n_p}(z_j)| + |f_{n_p}(z_j)
-f_{n_q}(z_j)|  + |f_{n_q}(z_j) -f_{n_q}(w)|<3\epsilon .\]
So, the subsequence $(f_{n_p})_{p\in\mathbb N}$ converges uniformly on
$K$. Claim 1 is proved.

Now we will show that there exists a subsequence of $(f_n)$ converging 
uniformly on every compact subset of $\Omega$.

{\sf Claim 2.} There exists an increasing sequence $(K_n)_{n\geq 1}$ of
compacta in $\Omega$, with each $K_n$ contained in the  interior of $K_{n+1}$, satisfying the condition  
\[\Omega=K_1 \cup K_2 \cup K_3 \cup \ldots.\] 

In fact, considering the following set, for each $n\in \mathbb N\setminus \{0\}$, 
\[K_n=\left\{z \in \overline{D}(0;n)\cap \overline{\Omega}: \, d(z;\partial
\Omega)\geq\frac{1}{n}\right\},\] 
it is clear that $K_n$ is closed and bounded and thus compact. Also, if $z\in K_n$ then $z\in \overline{\Omega}$ but $z\notin \partial \Omega$, implying $z \in \Omega$. Moreover, if $z \in K_n$, then we have $|z|\leq n<n+1$, $z\in \overline{\Omega}$, and $d(z;\partial \Omega)\geq \frac{1}{n}>\frac{1}{n+1}$. As a consequence, we obtain
\[K_n\subset \ \Omega_{n+1}=\Omega\cap D(0;n+1)\cap\left\{z: d(z;\partial
\Omega)>\frac{1}{n+1}\right\}\ \subset K_{n+1}\,,\] 
with $\Omega_{n+1}$ clearly open and $K_n$ in the interior of $K_{n+1}$. Claim 2 is proved.

To finish the proof of (b) we first notice that by applying Claim 1, we can choose
 an infinite set of indexes $I_1\subset \mathbb N$ such that the subsequence 
$(f_n)_{n\,\in\,  I_1}$, of the original sequence $(f_n)_{n \in \mathbb N}$, converges
on $K_1$. Moreover, for each $p\geq 2$, $p\in \mathbb N$, we can construct 
inductively an infinite set of indexes $I_p\subset I_{p-1}$ such that the
subsequence $(f_n)_{n \in I_p}$ converges on $K_p$.    

Next, applying the well-known ``Cantor's diagonal method'', we choose an
infinite set of indexes $I=\{i_1<i_2<...\}\subset \mathbb N$ such that $i_p\in I_p$, for all $p\geq 1$. Finally, from Claim 2 we conclude that the subsequence
$(f_{i_p})_{p\in \mathbb N}$ converges uniformly on every compact subset of $\Omega $.
\end{itemize}
\end{proof}

\section{Laurent Series}

Let us fix $r_1$ and $r_2$ such that $0\leq r_1<r_2\leq +\infty$. Let us
suppose that the power series $\sum_{n=0}^{+\infty}a_nz^n= a_0 + a_1z
+a_2z^2+ \cdots $ converges in the disk $\{z: |z|<r_2\}$ and also that the power series in the variable $\frac{1}{z}$, $\sum_{m=1}^{+\infty} a_{-m}z^{-m}= \frac{a_{-1}}{z} +  \frac{a_{-2}}{z^2} + \frac{a_{-3}}{z^3}+ \cdots $ converges inside $\{z: |z|>r_1\}$, the complement of the closed disk $\overline{D}(0;r_1)$. Then, we define the \textit{Laurent series} centered at zero $f(z)=\sum_{n=-\infty}^{+\infty}a_nz^n$ as 
\[f(z)=\sum\limits_{m=1}^{+\infty} a_{-m}z^{-m} +
\sum\limits_{n=0}^{+\infty}a_nz^n,\ \ \ \textrm{if}\ r_1<|z|<r_2\ .\]  
 
We say that the Laurent series $\sum_{n=-\infty}^{+\infty}a_nz^n$ converges on a subset $X$ of the annulus centered at zero $\{z: r_1<|z|<r_2\}$ if the series $\sum_{m=1}^{+\infty} a_{-m}z^{-m}$ and
$\sum_{n=0}^{+\infty}a_nz^n$ are both convergent for all $z\in X$. Hence, it
is easy to see that the Laurent series just defined converges for all $z$ satisfying $r_1<|z|<r_2$.

\begin{theorem}\label{TEO131} Let us suppose that $f(z)=\sum_{j=-\infty}^{+\infty}a_jz^j$, where
 $r_1<|z|< r_2$. If $r$ is such that $r_1<r<r_2$, then we have
\[\sum_{j=-\infty}^{+\infty}|a_j|^2\,r^{2j} \leq  M(r)^2,\  \textrm{where}\ M(r) =\max_{|z|=r}|f(z)|.\]
\end{theorem}
\begin{proof} Let us pick an arbitrary $z \in \mathbb C$ such that $|z|=r$. Given an arbitrary $N\in \mathbb N$, according to the triangle inequality we have 
\[\Big|\sum_{j=-N}^{j=N}a_jz^j\Big| \leq M(r)+
\Big|\sum_{j=N+1}^{+\infty}a_{-j}z^{-j}  \,+\,
\sum_{j=N+1}^{+\infty}a_jz^j\Big|.\] 
Since $z^j=r^{-N}z^{j+N}$, from the inequality right above it may be concluded that 
\[\Big|\sum_{j=-N}^{j=N}a_jr^{-N}z^{j+N}\Big|   \leq  M(r) +
\sum_{j=N+1}^{+\infty}|a_{-j}|r^{-j}  + \sum_{j=N+1}^{+\infty}|a_j|r^{j}
.\] 
Hence, by the Gutzmer-Parseval inequality for polynomials (Lemma ~\ref{LEM31}) we obtain
\[\sum_{j=-N}^{j=N}|a_j|^2r^{2j} \leq  \left(M(r) +
\sum_{j=N+1}^{+\infty}|a_{-j}|r^{-j} +
\sum_{j=N+1}^{+\infty}|a_j|r^j\right)^2.\]  
Taking the limit of the last inequality for $N\to +\infty$ yields the claimed inequality.
\end{proof}

Keeping the hypothesis in Theorem \ref{TEO131} we have the following result.

\begin{corollary}\label{132} If $f(z)=\sum_{j=-\infty}^{+\infty}a_jz^j= 0$, where $r_1<|z|< r_2$, then we have $a_j=0$, for all $j\in \mathbb Z$.
\end{corollary}
\begin{proof} It follows straightforward from Theorem \ref{TEO131}.
\end{proof}

\

\paragraph{\bf Acknowledgments.} The author wishes to express his gratitude to Professor R. B. Burckel for his active interest in the publication of this paper, many references and stimulating emails, and very helpful suggestions. I am also very  thankful to Professors
J. V. Ralston and Paulo A. Martin for their comments and suggestions. The possible slips and mistakes are my sole responsibility.

\bibliographystyle{amsplain}

\end{document}